\documentclass[11pt]{amsart}
\usepackage{amsmath, amscd, amsthm, amssymb}
\usepackage{graphicx}
\usepackage{enumitem}

\usepackage{fullpage}

\numberwithin{equation}{section}

\def\C{{\mathbf C}}
\def\R{{\mathbf R}}
\def\Z{{\mathbf Z}}
\def\Q{{\mathbf Q}}

\def\A{{\mathbf A}}
\def\F{{\mathbf F}}

\def\O{{\mathcal O}}

\def\blank{\phantom{x}}

\newtheorem{theorem}{Theorem}[section]

\newtheorem{lemma}[theorem]{Lemma}

\newtheorem{proposition}[theorem]{Proposition}

\newtheorem{corollary}[theorem]{Corollary}

\theoremstyle{definition}
\newtheorem{definition}[theorem]{Definition}

\theoremstyle{remark}
\newtheorem{remark}[theorem]{Remark}

\DeclareMathOperator{\diag}{diag}

\DeclareMathOperator{\charf}{char}
\renewcommand{\(}{\left(}
\renewcommand{\)}{\right)}

\newcommand{\mm}[4]{\(\begin{smallmatrix} #1 & #2\\ #3 & #4\end{smallmatrix}\)}

\DeclareMathOperator{\tr}{tr}

\DeclareMathOperator{\Sp}{Sp}
\DeclareMathOperator{\GSp}{GSp}
\DeclareMathOperator{\PGSp}{PGSp}
\DeclareMathOperator{\GU}{GU}

\DeclareMathOperator{\SL}{SL}
\DeclareMathOperator{\GL}{GL}

\DeclareMathOperator{\iIm}{Im}

\DeclareMathOperator{\GG}{G}
\DeclareMathOperator{\WW}{W_{\Sp_6}}
\DeclareMathOperator{\GS}{\widetilde{G}}

\begin{document}
\title{The Spin $L$-function on $\GSp_6$ for Siegel modular forms}
\author{Aaron Pollack}
\address{Department of Mathematics\\ Stanford University\\ Stanford, CA USA}
\email{aaronjp@stanford.edu}

\begin{abstract}
We give a Rankin-Selberg integral representation for the Spin (degree eight) $L$-function on $\PGSp_6$ that applies to the cuspidal automorphic representations associated to Siegel modular forms.  If $\pi$ corresponds to a level one Siegel modular form $f$ of even weight, and if $f$ has a non-vanishing \emph{maximal} Fourier coefficient (defined below), then we deduce the functional equation and finiteness of poles of the completed Spin $L$-function $\Lambda(\pi,Spin,s)$ of $\pi$.
\end{abstract}

\subjclass[2010]{Primary 11F66, 11F46; Secondary 11F03, 22E55}

\maketitle

\section{Introduction}
\subsection{Main theorems} The purpose of this paper is to give a family of Rankin-Selberg integrals $I_B(\phi,s)$ for the Spin $L$-function of the cuspidal automorphic representations of $\GSp_6(\A)$ that correspond to Siegel modular forms.  The Rankin-Selberg integrals are indexed by quaternion algebras $B$ over $\Q$ that are ramified at the infinite place.

Denote by $\mathcal H_3$ the Siegel upper half-space of degree three.  $\mathcal H_3$ is the space of three-by-three complex symmetric matrices $Z$ whose imaginary part is positive definite:
\[\mathcal H_3 = \{Z = X+i Y \in M_3(\C): \,^t Z = Z, Y > 0\}.\]
Denote by $\GSp_6^+(\R)$ the subgroup of $\GSp_6(\R)$ whose elements have positive similitude.  Then $\GSp_6^+(\R)$ acts on $\mathcal H_3$ by the formula $g Z = (AZ+B)(CZ+D)^{-1}$ if $g = \mm{A}{B}{C}{D}$. Define $K_{\infty,\Sp_6} \simeq U(3)$ to be the subgroup of $\Sp_6(\R)$ that stabilizes $i \in \mathcal{H}_3$ for this action.  For $Z \in \mathcal H_3$ and $g \in \GSp_6^+(\R)$, define $j(g,Z) = \nu(g)^{-2}\det(CZ+D)$, where $\nu$ is the similitude. 

Suppose $r \geq 0$ is an integer, and $\pi$ is a cuspidal automorphic representation of $\GSp_6(\A)$.  
\begin{definition}\label{holAssumption} We say that $\pi$ is associated to a level one Siegel modular form of weight $2r$ if the following conditions are satisfied:
\begin{itemize}
\item There is a nonzero $\phi$ in the space of $\pi$ satisfying $\phi(gk_f k_\infty) = j(k_\infty,i)^{-2r}\phi(g)$ for all $g \in \GSp_6(\A)$, $k_f \in \prod_{p}{\GSp_6(\Z_p)}$, and $k_\infty \in K_{\infty,\Sp_6}$.
\item $\pi$ has trivial central character.
\item The function $f_\phi: \mathcal H_3 \rightarrow \C$ well-defined by the equality $f_\phi(g_\infty(i)) = \nu(g_\infty)^r j(g_\infty,i)^{2r}\phi(g_\infty)$ for all $g_\infty \in \GSp_6(\R)^+$ is a classical Siegel modular form of weight $2r$.
\end{itemize}
\end{definition}
See, for example, \cite[Section 4]{asSch} for more on the relationship between Siegel modular forms and their associated automorphic representations.  The content of the final condition is that $f_\phi$ is a holomorphic function on $\mathcal H_3$, and thus has a Fourier expansion of the form
\[f_\phi(Z) = \sum_{T > 0}{a_f(T)e^{2\pi i \tr(TZ)}}.\]
Here the sum ranges over all half-integral, symmetric, three-by-three matrices $T$ that are positive definite.

Via the even Clifford construction, half-integral, symmetric, three-by-three matrices $T$ correspond to pairs $(\Lambda, \{1,v_1, v_2, v_3\})$ consisting of orders $\Lambda$ in quaternion algebras over $\Q$ equipped with a particular type of nice ordered basis $\{1, v_1, v_2, v_3\}$ of $\Lambda$.  This correspondence goes back at least to Brzezinski \cite[Proposition 5.2]{brzezinski}, and more recently it has been studied and generalized by Gross-Lucianovic \cite{gL}. The condition that $T$ be positive (or negative) definite is exactly that the associated quaternion algebra $\Lambda \otimes \Q$ be ramified at infinity.  As an example, the identity matrix $1_3$ corresponds to the order $\Z + \Z i +\Z j + \Z k$ in Hamilton's quaternions, while the matrix
\[\left(\begin{array}{ccc} 1 & \frac{1}{2} & \frac{1}{2} \\ \frac{1}{2} & 1 & \frac{1}{2} \\ \frac{1}{2} & \frac{1}{2} & 1 \end{array} \right)\]
corresponds to Hurwitz's maximal order 
\[\left\{a + xi + yj + zk: a,x, y, z \in \Z \text{ or } a,x,y,z \in \Z + \frac{1}{2}\right\}.\]

A cuspidal representation $\pi$ that is associated to a level one Siegel modular form is spherical at all finite places, and thus has a Spin $L$-function $L(\pi,Spin,s)$, defined for $Re(s) >>0$ as an absolutely convergent Euler product over all the primes of $\Q$.  If $\pi$ is associated to Siegel modular form of weight $2r$, we define the completed Spin $L$-function $\Lambda(\pi,Spin,s)$ to be
\[\Lambda(\pi, Spin, s) = \Gamma_\C(s+r-2)\Gamma_\C(s+r-1)\Gamma_\C(s+r)\Gamma_\C(s+3r-3)L(\pi,Spin,s)\]
where, as usual, $\Gamma_\C(s) = 2(2\pi)^{-s} \Gamma(s)$.  Our main theorem is the following.
\begin{theorem}\label{MainThm} Suppose that $\pi$ is a cuspidal automorphic representation of $\GSp_6(\A)$ associated to a level one Siegel modular form $f$ of weight $2r$.  Suppose furthermore that $a_f(T) \neq 0$ for some Fourier coefficient $a_f(T)$ of $f$, such that $T$ corresponds to a \emph{maximal} order in some quaternion algebra over $\Q$ ramified at infinity.  Then $\Lambda(\pi,Spin,s)$ has meromorphic continuation in $s$ to the entire complex plane with at worst finitely many poles, and satisfies the functional equation $\Lambda(\pi,Spin,s) = \Lambda(\pi,Spin,1-s)$.\end{theorem}

The meromorphic continuation of the $L$-function, but not the functional equation or finiteness of the poles, was known previously.  Namely, the Spin $L$-function on $\GSp_6$ appears in the constant term of a cuspidal Eisenstein series on an exceptional group of type $F_4$.  Thus its meromorphic continuation follows from the method of Langlands \cite{LanglandsEP}, see Shahidi \cite[Theorem 6.1]{shahidi}.  Since the underlying cuspidal representation of a Siegel modular form is not generic, the Langlands-Shahidi method cannot be used to obtain the functional equation or the finiteness of the poles of this $L$-function.

The bound we get on the poles is that they are contained in the set of integer and half-integer points in the interval $[-2,3]$; this bound is not expected to be optimal.  One should be able to get much better control of the poles by applying the theorems of Yamana \cite{yamana2}. Let us also remark that it is known \cite{yamana1} that every level one Siegel modular form has a non-vanishing Fourier coefficient $a(T)$ for a $T$ that is \emph{primitive}, meaning that $T/n$ is not half-integral for all integers $n > 1$.  The primitive $T$ correspond to \emph{Gorenstein} quaternion orders \cite{gL}, the set of which strictly contain the set of maximal orders.  

It is an interesting question to determine if every level $1$ Siegel modular form $f$ on $\GSp_6$ as in Theorem \ref{MainThm} has a non-vanishing Fourier coefficient $a_{f}(T)$ with $T$ corresponding to a maximal order.  The analogous statement for Siegel modular forms on $\GSp_4$ was proved by Saha in \cite{saha}.  On $\GSp_4$, Fourier coefficients $a(T)$ of level $1$ Siegel modular forms are parametrized by $2 \times 2$ half-integral symmetric matrices.  Applying Gauss composition, such $T$ correspond to pairs $(S,I)$ where $S$ is a quadratic ring and $I$ is an ideal class for $S$.  It is a consequence of the main results of \cite{saha} that every level $1$ Siegel modular form on $\GSp_4$ has a non-vanishing Fourier coefficient $a(T)$ for a $T$ that corresponds to a pair $(S,I)$ with $S$ the maximal order in a quadratic imaginary extension of $\Q$.  Analogously to Theorem \ref{MainThm}, the need for a non-vanishing ``maximal" Fourier coefficient for Siegel modular forms on $\GSp_4$ arises in the integral representation of Furusawa \cite{furusawa}.

Since we obtain the functional equation of the completed Spin $L$-function for level one Siegel modular forms, we thus get the exact order of vanishing of $L(\pi,Spin,s)$ for these $\pi$ at sufficiently negative integers.
\begin{corollary} Suppose that $\pi$ is a cuspidal automorphic representation of $\GSp_6(\A)$ associated to a level one Siegel modular form $f$ of even weight, and that $a_f(T) \neq 0$ for a Fourier coefficient of $f$ corresponding to a maximal order in a quaternion algebra over $\Q.$  Then $ord_{s=-j}L(\pi,Spin,s) = 4$ for sufficiently large integers $j$. \end{corollary}

Theorem \ref{MainThm} follows from the analysis of the Rankin-Selberg integral $I_B(\phi,s)$ for the Spin $L$-function, together with the analytic properties of the normalized Eisenstein series used in the integral.  Suppose that $B$ is a quaternion algebra over $\Q$, ramified at infinity.  Denote by $J$ the set of three-by-three Hermitian matrices over $B$: $J = \{ h \in M_3(B): h^* = h\},$ where for $x \in M_3(B)$, $x^*$ denotes the transpose conjugate of $x$, the conjugation being from the quaternion structure on $B$.  Associated to $J$, the \emph{Freudenthal construction} produces a reductive $\Q$-group $\GG$ of type $D_6$, and there is an embedding $\GSp_6 \rightarrow \GG$.  The group $\GG$ has a maximal parabolic $P$, whose unipotent radical is abelian, and naturally isomorphic to $J$ under addition.  Suppose $r \geq 0$ is an integer, and $B_0$ is a maximal order in $B$.  Associated to $r, B_0$, and a character $\chi_s: P(\A) \rightarrow \C^\times$, we define a normalized Eisenstein series $E_{2r}^*(g,s)$ for $\GG$. 

For a cusp form $\phi$ in the space of a cuspidal representation $\pi$ associated to a Siegel modular form of weight $2r$, the Rankin-Selberg integral is
\[I_{2r}(\phi,s) = \int_{\GSp_6(\Q)Z(\A)\backslash \GSp_6(\A)}{\phi(g)E_{2r}^*(g,s)\,dg}.\]
Suppose that the maximal order $B_0$ corresponds to the half-integral symmetric matrix $T$.  When $\phi$ is as in Definition \ref{holAssumption}, we show that $I_{2r}(\phi,s)$ is equal to $a_{f_\phi}(T) \Lambda(\pi,Spin,s-2)$, up to a nonzero constant.  To show this equality, we unfold the integral $I_{2r}(\phi,s)$, and then apply the result \cite{evdokimov} of Evdokimov, who computed a Dirichlet series for the Spin $L$-function on $\GSp_6$.  More precisely, \cite[Theorem 3]{evdokimov} is not quite a Dirichlet series for the Spin $L$-function.  However, when one combines Theorem 3 of loc. cit. with Lemma \ref{AITlemma} below, which involves the arithmetic invariant theory of quaternion orders, and the fact that $B_0$ is a maximal order, one does obtain a Dirichlet series for the Spin $L$-function.  The integral $I_{2r}(\phi,s)$ then unfolds to exactly this Dirichlet series.

In section \ref{Eisenstein}, we use Langlands' theorem of the constant term and functional equation, together with Maass-Shimura differential operators, to analyze the analytic properties of the normalized Eisenstein series $E_{2r}^*(g,s)$.  We prove that it satisfies the functional equation $E_{2r}^*(g,s) = E^*_{2r}(g,5-s)$ and has at most finitely many poles, all contained in the set of integral and half-integral points in the interval $[0,5]$.  Theorem \ref{MainThm} then follows.

\subsection{The Spin $L$-function} When a reductive $\Q$-group has an associated Shimura variety, the minuscule cocharacter $\mu$ that is part of the Shimura datum singles out a representation of the Langlands dual group, and thus an automorphic $L$-function, that is of particular importance.  These $L$-functions are expected to control the arithmetic of the Shimura variety.  On Siegel modular varieties this is the Spin $L$-function.  For Siegel modular forms on $\GSp_4$, the desired analytic properties of this $L$-function were first worked out by Andrianov \cite{andrianov2A}, \cite{andrianov2B}.  Other Rankin-Selberg integrals for spinor $L$-functions that apply to holomorphic modular forms are the triple product of Garrett \cite{garrett} on $\GL_2 \times \GL_2 \times \GL_2$, which was extended to $\GL_2$ over a general \'etale cubic algebra by Piatetski-Shapiro and Rallis \cite{psrRankin3}, and the integral of Furusawa \cite{furusawa} for the Spin $L$-function on $\GSp_4 \times \GU(1,1)$.  The Spin $L$-function of Siegel modular forms on $\GSp_{2n}$ for $n >3$ remains poorly understood.

For the Spin $L$-function on $\GSp_6$, there is also the integral of Bump-Ginzburg \cite{bg} that applies to generic cusp forms, and more recently the result \cite{pollackShah} of Pollack-Shah that applies to cuspidal representations that support a rank two Fourier coefficient.  The integrals in \cite{bg} and \cite{pollackShah} vanish identically for holomorphic Siegel modular forms.  

It turns out that the integral in this paper, and the older integrals \cite{garrett}, \cite{psrRankin3}, \cite{furusawa}, \cite{bg}, and \cite{pollackShah} all appear to be mysteriously connected to the magic triangle of Deligne and Gross \cite{dG}, as we explain in Appendix \ref{magic}.  This connection seems to be closely related to the so-called \emph{towers} of Rankin-Selberg integrals introduced by Ginzburg-Rallis \cite{ginzRallis}.  Furthermore, that Freudenthal's magic square might be connected to the Rankin-Selberg method is discussed, from a perspective different from our own, in \cite{bumpRS}.  Nevertheless, we hope the examples we discuss in the appendix will stimulate future inquiry. 

\subsection{Layout of paper} The contents of the paper is as follows.  In section \ref{sec:Freud} we review the Freudenthal construction.  In section \ref{AIT}, we review and slightly extend the results of \cite{gL} that relate ternary quadratic forms and quaternion orders.  In section \ref{sec:dirichlet}, we recall \cite[Theorem 3]{evdokimov} of Evdokimov, and combine this result with the work of section \ref{AIT} to obtain a Dirichlet series for the Spin $L$-function.  We define the global Rankin-Selberg integral and unfold it in section \ref{sec:Global}.  In section \ref{unram} we evaluate the unfolded integral in terms of $\Lambda(\pi,Spin,s)$.  We give some results on the group $\GG(\A)$ and analyze the normalized Eisenstein series $E_{2r}^*(g,s)$ in section \ref{Eisenstein}.  The group $\GG$ is closely related to a classical group $\GU_6(B)$.  In Appendix \ref{GU6} we make this relationship explicit.  In Appendix \ref{magic} we briefly discuss the magic triangle of Deligne and Gross and its connection to some Rankin-Selberg integrals.

\subsection{Acknowledgments} It is a pleasure to thank Shrenik Shah for his collaboration on \cite{pollackShah},  without which this work might never have come to fruition.  We are grateful to Benedict Gross, Mark Lucianovic, Ila Varma, and Akshay Venkatesh for enlightening conversations.  We also thank Siegfried B\"ocherer, Gordan Savin and John Voight for their very helpful comments on an earlier version of this manuscript.  This work has been partially supported through the NSF by grant DMS-1401858.

\section{The Freudenthal construction}\label{sec:Freud}
In this section we review the Freudenthal construction \cite{freudenthal}, specialized to the case of interest for us.  Nothing in this section is new.  We follow the exposition of Bhargava-Ho \cite[Sections 5.3 and 6.4]{bhargavaHo} fairly closely, with the exception of our definition of rank one elements.  The reader is also instructed to see Springer \cite{springer} and the references contained therein for more on the Freudenthal construction.
\subsection{Three by three hermitian matrices}
Suppose that $B$ is a quaternion algebra over the ground field $F$ of characteristic zero.  For $a \in B$, write $a^*$ for the conjugate of $a$, $\tr(a) = a + a^*$, $n(a) = aa^*$, the trace and norm of $a$.  These are considered as elements of $F$.  The symmetric pairing on $B$ is defined by 
\begin{equation}\label{sym B pairing}(u,v) = \tr(uv^*) = (v,u). \end{equation}

We consider $3 \times 3$ hermitian matrices over $B$.  If $h$ is such a matrix, then 
\begin{equation}\label{h def}h = \left(\begin{array}{ccc}c_1 & a_3 & a_2^*\\ a_3^* & c_2 & a_1 \\ a_2 & a_1^* & c_3\end{array}\right)\end{equation}
where $c_1, c_2, c_3$ are elements of $F$ and $a_1, a_2, a_3$ are elements of $B$.  The set of such matrices is denoted $H_3(B)$.

For $h$ as in (\ref{h def}), define
\[N(h) := c_1c_2c_3 - c_1n(a_1) - c_2n(a_2) - c_3n(a_3) + \tr(a_1a_2a_3)\]
the norm of $h$, and 
\[\tr(h) = c_1 + c_2 + c_3\]
the trace of $h$.  Also define
\begin{equation}\label{hSharp} h^{\#} := \left(\begin{array}{ccc}c_2c_3-n(a_1) & a_2^*a_1^*-c_3a_3 & a_3a_1-c_2a_2^*\\ a_1a_2-c_3a_3^* & c_1c_3-n(a_2) & a_3^*a_2^*-c_1a_1 \\ a_1^*a_3^*-c_2a_2 & a_2a_3-c_1a_1^* & c_1c_2-n(a_3)\end{array}\right).\end{equation}
One has that $hh^\# = N(h)1_3 = h^\#h$, and $(h^\#)^\# = N(h)h$.  
\begin{definition}\label{rnkH3B} An element $h \in H_3(B)$ is said to be of \emph{rank three} if $N(h) \neq 0$, of \emph{rank two} if $N(h) = 0$ but $h^\# \neq 0$, and of \emph{rank one} if $h^\# = 0$ but $h \neq 0$. \end{definition}

For $x, y$ three-by-three hermitian matrices over $B$, define
\[x \times y := (x+y)^\# - x^\# - y^\#\]
and
\[ \tr(x,y) := \frac{1}{2}\tr(xy+yx).\]
Finally, denote by
\[(\blank, \blank,\blank): H_3(B) \times H_3(B) \times H_3(B) \rightarrow F\]
the unique symmetric trilinear form satisfying $(x,x,x) = 6N(x)$.  Equivalently, 
\[(x,y,z) = N(x+y+z) - N(x+y)-N(x+z)-N(y+z) + N(x) + N(y) + N(z).\]
For $x, y, z \in H_3(B)$, one has
\[\tr(x \times y, z) = (x,y,z) = \tr(x, y\times z).\] 

The trilinear form is nondegenerate in the sense that if $(x,y,z) = 0$ for all $x, y$, then $z = 0$.  The formula
\[\tr(x,y) = \sum_{1 \leq i \leq 3}{x_{ii}y_{ii}} + \sum_{1\leq i < j \leq 3}{(x_{ij},y_{ij})}\]
shows that the bilinear form $\tr(\blank, \blank)$ is nondegenerate.  Here $x_{ij}$ denotes the $i, j$ entry of $x$, and $(\blank, \blank)$ is the symmetric pairing on $B$ defined in (\ref{sym B pairing}).  For $x, y \in H_3(B)$, one has the formula
\begin{equation}\label{Nx+y} N(x+y) = N(x) + \tr(x^\#,y) + \tr(x,y^\#) + N(y).\end{equation}
\subsection{The defining representation}
Fix the quaternion algebra $B$ over $F$.  We denote by $W$ the vector space $F \oplus H_3(B) \oplus H_3(B) \oplus F$.   Typical elements of $W$ are denoted $(a,b,c,d)$ with $a, d \in F$, $b,c \in H_3(B)$.  The space $W$ is equipped with a symplectic and quartic form.  The symplectic form $\langle \blank, \blank \rangle : W \times W \rightarrow F$ is defined by 
\[ \langle (a,b,c,d) , (a',b',c',d') \rangle = ad' - \tr(b,c') + \tr(c,b') - da'.\]
The quartic form $Q: W \rightarrow F$ is defined by 
\[Q((a,b,c,d)) = (ad - \tr(b,c))^2 + 4aN(c) + 4dN(b) - 4\tr(b^\#,c^\#).\]

We now define the (similitude) algebraic group $\GG$.
\begin{definition} Consider $\GL(W)$ to act on the right of $W$.  The group $\GG$ is defined to be the set of $(g,\nu(g)) \in \GL(W) \times \GL_1$ that satisfy $\langle ug , vg \rangle = \nu(g) \langle u, v \rangle$ and $Q(wg) = \nu(g)^2 Q(w)$ for all $u, v, w \in W$.  The map $\nu: \GG \rightarrow \GL_1$ sending $g \mapsto \nu(g)$ is called the \emph{similitude}. \end{definition}
Typically, the Freudenthal construction defines the subgroup of $\GG$ with similitude 1, but we will need this similitude version.  It is clear that $\GG$ is a reductive group; knowledge of the generators of $\GG$ proves that it is connected as well \cite{brown}.

We will give some explicit elements of $\GG$. Most simply, we have the following maps:
\begin{itemize}
\item For $\lambda \in \GL_1$, $(a,b,c,d) \mapsto (\lambda a, \lambda b, \lambda c, \lambda d)$, with similitude $\lambda^2$
\item and $(a,b,c,d) \mapsto (\lambda^2 a, \lambda b, c, \lambda^{-1} d)$, with similitude $\lambda$;
\item $(a,b,c,d) \mapsto (-d,c,-b,a)$ with similitude 1.
\end{itemize}
This last map we denote by $J_6$, or $J$, if no confusion seems likely.  Under the map $\GSp_6 \rightarrow \GG$ defined in section \ref{gsp6toG}, $J_6$ is the image of the element $\mm{}{1_3}{-1_3}{}$ of $\GSp_6$.

Here are some more interesting elements of $\GG$.  First, for $X \in H_3(B)$, define $n(X): W \rightarrow W$ via the formula
\begin{equation}\label{n(X)form} (a,b,c,d)n(X) = (a, b + aX, c+ b \times X + aX^\#, d + \tr(c,X) + \tr(b,X^\#) + aN(X)).\end{equation}
Similarly, for $Y \in H_3(B)$, define $\bar{n}(Y): W \rightarrow W$ via the formula
\[(a,b,c,d)\bar{n}(Y) = (a + \tr(b,Y) + \tr(c,Y^\#) + dN(Y), b+c\times Y + dY^\#, c + dY, d).\]
The following lemma says that the $n(X), \bar{n}(Y)$ are elements of $\GG$ with similitude 1, that $n(X_1 + X_2) = n(X_1)n(X_2)$, and similarly for $\bar{n}$.  For a proof of the lemma, see \cite[Lemma 2.3]{springer}. 
\begin{lemma}\label{n(X)} The maps $X \mapsto (n(X),1)$ and $Y \mapsto (\bar{n}(Y),1)$ define group homomorphisms $H_3(B) \rightarrow \GG$.\end{lemma}
We define one more type of (fairly) explicit element of $\GG$.  Consider triples 
\[(t, \tilde{t}, \lambda) \in \GL(H_3(B)) \times \GL(H_3(B)) \times \GL_1\]
that satisfy, for all $b, c \in H_3(B)$
\begin{itemize}
\item $\tr(t(b),\tilde{t}(c)) = \tr(b,c)$
\item $N(t(b)) = \lambda N(b)$
\item $N(\tilde{t}(c)) = \lambda^{-1} N(c)$
\item $t(b)^\# = \lambda \tilde{t}(b^\#)$ and $\tilde{t}(c)^\# = \lambda^{-1}t(c^\#)$
\end{itemize}
Since the pairing $\tr(\;,\;)$ is nondegenerate, the first property uniquely determines $\tilde{t}$ in terms of $t$, and then the first two properties imply the final two. If $m \in \GL_3(B)$, and $r \in \GL_1(F)$, then the linear map $t(b) = r^{-1} m^* b m$ satisfies $N(t(b)) = r^{-3}N(m^* m)N(b)$ for all $b$ in $H_3(B)$, and thus defines such an automorphism.  For such a triple $(t , \tilde{t}, \lambda)$, one defines a map $m(t): W \rightarrow W$ via
\[ (a,b,c,d)m(t) = (\lambda a, t(b), \tilde{t}(c), \lambda^{-1} d).\]
It is immediate that $m(t)$ defines an element of $\GG$ with similitude $1$.

We now define the rank of an element of $W$.  First, polarizing the quartic form $Q$, there is a unique symmetric $4$-linear form on $W$ that satisfies $(v,v,v,v) = Q(v)$ for all $v$ in $W$.
\begin{definition} All elements of $W$ have rank $\leq 4$.  An element $v$ of $W$ is of rank $\leq 3$ if $Q(v) = (v,v,v,v) = 0$.  An element $v$ is of rank $\leq 2$ if $(v,v,v,w) = 0$ for all $w$ in $W$.  An element $v$ is of rank $\leq 1$ if $(v,v,w,w') = 0$ for all $w$ in $W$ and $w'$ in $(Fv)^\perp$.  Here $(Fv)^\perp$ is the set of $x \in W$ with $\langle v, x \rangle = 0$.  Finally, $0$ is the unique element of rank $0$. Since $\GG$ preserves the symplectic and the $4$-linear form, $\GG$ preserves the rank of elements of $W$.\end{definition}

Elements of $W$ of rank one will play an important role in this paper.  For example, $f = (0,0,0,1)$ is of rank one, and then applying $\overline{n}(X)$ for $X$ in $H_3(B)$, one gets $(N(X),X^\#,X,1)$ is rank one.  To check that $f$ is rank one, the reader must simply observe (immediately) that the coefficient of $\mu \mu'$ in $Q(f + \mu w + \mu' w')$ is zero for general $w = (a,b,c,d)$ and $w' = (0,b',c',d')$ in $W$.

It is easy to check that all the elements of rank one are in a single $\GG$ orbit.  For example, one argument goes as follows: First, by applying operators $n(X)$, $J_6$, and scalar multiplication, it is easy to see that every element $v$ of $W$ can be moved to one with $d=1$.  Then applying $\overline{n}(X)$, every element has a $\GG$ translate of the form $v_0 = (z,y,0,1)$.  Now, if $v$ and thus $v_0$ is rank one, then $(v_0,v_0,w,w') = 0$ for arbitrary $w = (0,b,c,0)$, $w' = (0,b',0,0)$.  The coefficient of $2 \mu \mu'$ in $Q(v_0 + \mu w + \mu w')$ is now easy to compute, and it is found to be $2\tr(y,b\times b') -z\tr(b',c)$.  Since this must be zero for all $b, b', c$, it follows that $y$ and $z$ are zero, by the nondegeneracy of the two pairings.  Hence $v_0 = (0,0,0,1) = f$.

If $v$ is a rank one element, the stabilizer in $\GG$ of the line $Fv$ is a parabolic subgroup of $\GG$.  Indeed, consider the set of flags of the form $0 \subseteq \ell \subseteq W' \subseteq W$ where $\ell$ is dimension one and $W'$ is codimension one.  The conditions $\langle v, w' \rangle =0$ and $(v,v,w,w') = 0 $ for all $v \in \ell$, $w' \in W'$ and $w \in W$ clearly form closed conditions on the Grassmanian of flags of these dimensions.  Since these conditions single out the rank one lines, and since $\GG$ acts transitively on these lines, the stabilizer of a rank one line is a parabolic subgroup.

The argument above also shows that if $(z,y,x,1)$ is rank one, then $y = x^\#$ and $z = N(x)$. (Just apply $\overline{n}(-x)$ to $(z,y,x,1)$.) In fact, the following is true.
\begin{proposition}\label{rk1Sharp} The element $(a,b,c,d)$ has rank at most one if and only if $b^\# = ac, c^\# = db$, $bc = cb$, and $ad = bc$.  Here, the products $bc$ and $cb$ are taken in $M_3(B)$. \end{proposition}
A proof of the proposition may be found in \cite[Proposition 11.2]{ganSavin}.

\section{Arithmetic invariant theory}\label{AIT}
In this section we review the arithmetic invariant theory of Gross-Lucianovic \cite{gL}, and extend it slightly.  The results in \cite{gL} relate ternary quadratic forms to orders in quaternion algebras, together with a nice basis of this order, up to isomorphism. The reader should consult Gross-Lucianovic \cite{gL} and Voight \cite{voight} for more details and related results.  

The relation between ternary quadratic forms and quaternion algebras has a long history, going back at least to Brandt \cite{brandt}, Eichler \cite{eichler}, and Brzezinksi \cite{brzezinski, brzezinski2}.  It has been studied by many others as well, and from varying perspectives.  The formulation of \cite{gL} turns out to be very well-suited for our analysis of the Rankin-Selberg integral.

First, following \cite{gL}, a quaternion ring is defined as follows. Suppose $R$ is a commutative ring, and $\Lambda$ is an associative $R$-algebra with unit.  Then $\Lambda$ is a \emph{quaternion ring} over $R$ if the following properties are satisfied:
\begin{itemize}
\item $\Lambda$ is free of rank four as an $R$-module, and $\Lambda/(R \cdot 1)$ is free of rank three.
\item There is an $R$-linear anti-involution $x \mapsto x^*$ on $\Lambda$, such that $x^* = x$ for $x$ in $R$, and $\tr(x):=x + x^*$, $n(x) := x x^*$ are in $R$. 
\item If $x \in \Lambda$, then left multiplication by $x$ on $\Lambda$ has trace $2\tr(x)$. \end{itemize}
For instance then, if $x \in \Lambda$, then $x^2 = \tr(x)x-n(x)$.  Note that we always assume our quaternion rings are free; a vast generalization is considered in \cite{voight}.  The following definition and lemma from \cite{gL} are crucial.
\begin{definition}[\cite{gL}] Suppose $1, v_1, v_2, v_3$ form a basis of the quaternion ring $\Lambda$.  Then $1, v_1, v_2, v_3$ are said to form a \emph{good basis} if there exist $a, b,c$ in $R$ such that the matrix
\begin{equation}\label{A matrix}A = \left(\begin{array}{ccc} a & v_3 & v_2^* \\ v_3^* & b & v_1 \\ v_2 & v_1^* & c \end{array}\right)\end{equation}
is rank one in the sense of Definition \ref{rnkH3B}.  Equivalently, the basis $1, v_1, v_2, v_3$ is a good basis if there exists $a,b, c$ in $R$ such that
\begin{equation}\label{gB1} v_2 v_3 = a v_1^*; \qquad v_3 v_1 = b v_2^*; \qquad v_1 v_2 = c v_3^*; \end{equation}
and
\begin{equation}\label{gB2} n(v_1) = bc; \qquad n(v_2) = ca; \qquad n(v_3) = ab.\end{equation}
\end{definition}
The formulas (\ref{gB1}) and (\ref{gB2}) are how good bases are defined in \cite{gL}.  That one can equivalently define good bases in terms of the $A$ of (\ref{A matrix}) being rank one is suggested by \cite[Proposition 2.5.2]{lucianovic}.  Gross-Lucianovic prove:
\begin{lemma}[\cite{gL}] \label{uniquexi} If $1, w_1, w_2, w_3$ are a basis of $\Lambda$, then there exist unique $x_1, x_2, x_3$ in $R$ so that $v_i:=w_i - x_i$ form a good basis. \end{lemma}
Let us very briefly sketch a proof.  Consider the structure coefficients of $\Lambda$, namely, those elements $s_{ij}^k$ in $R$ so that
\[w_i w_j = s_{ij}^0 + s_{ij}^1 w_1 + s_{ij}^2 w_2 + s_{ij}^3 w_3.\]
The key point is that the definition of the quaternion ring forces the matrix
\[S:=\left(\begin{array}{ccc} s_{23}^1 & s_{23}^2 & s_{23}^3 \\ s_{31}^1 & s_{31}^2 & s_{31}^3 \\ s_{12}^1 & s_{12}^2 & s_{12}^3 \end{array}\right)\]
to be \emph{symmetric}.  For example, to see that $s_{23}^2 = s_{31}^1$, one considers the trace of left multiplication by $w_3$ on $\Lambda$.  Using the expression for $w_2 w_3$ in terms of the structure constants, and conjugating it, one can obtain that the trace is $2\tr(w_3) + s_{31}^1-s_{23}^2$.  Consequently, $s_{31}^1 = s_{23}^2$.  One sets $x_1 = s_{31}^3 = s_{12}^2$, $x_2 = s_{12}^1 = s_{23}^3$, and $x_3 = s_{23}^2 = s_{31}^1$.  Then $v_i = w_i - x_i$ is the good basis.  In fact, if
\[W = \left(\begin{array}{ccc} & w_3 & w_2^* \\ w_3^* & & w_1 \\ w_2 & w_1^* & \end{array}\right),\]
then using the associativity of the multiplication in $\Lambda$ and that $1, v_1, v_2, v_3$ is a basis, one can check that $W-S$ is rank one.  The uniqueness of the $x_i$ is immediate.  For a slightly different proof, but more details, consult \cite{gL}. 

Suppose now that $v_1, v_2, v_3$ is a good basis, and $A$ is the rank one matrix in (\ref{A matrix}).  Associated to $A$, one can form the ternary quadratic form 
\[q_T(x,y,z) = ax^2 + by^2 + cz^2 + dyz+ ezx + fxy,\]
where $d = \tr(v_1)$, $e = \tr(v_2)$, $f = \tr(v_3)$.  Equivalently, if $2$ is not a zero-divisor in $R$, one can form the half-integral symmetric matrix
\begin{equation}\label{Tentries} T = \left(\begin{array}{ccc} a & \frac{f}{2} & \frac{e}{2} \\ \frac{f}{2} & b & \frac{d}{2} \\ \frac{e}{2} & \frac{d}{2} & c \end{array}\right).\end{equation}
The strength of the definition of good basis is that $q_T$ determines all the multiplication rules in the quaternion ring.  For instance, $v_1^2 = d v_1 - bc$, $v_1 v_2 = c v_3^* = c f - cv_3$, and
\[ v_2 v_1 = (e - v_2)(d-v_1) - de + ev_1 + dv_2 = - de + e v_1 + d v_2 + cv_3.\]

Conversely, starting with a ternary quadratic form $q_T(x,y,z) = ax^2 + by^2 + cz^2 + dyz + ezx + fxy$ on $R^3$, there is the data of a quaternion ring $\Lambda_T$ over $R$, together with a good basis of $\Lambda_T$, as follows.  Let us write $N = R^3$, with basis $e_1 = (1,0,0)$, $e_2 = (0,1,0), e_3 = (0,0,1)$.  One puts on $N$ the quadratic form
\[q_T(xe_1 + ye_2 + ze_3) = ax^2 + by^2 + cz^2 + dyz + ezx + fxy.\]
Then, one can construct the even Clifford algebra $C^+((N,q_T))$, defined as usual to be the even degree part of the tensor algebra of $N$, modulo the ideal generated by the relations $v^2 = q_T(v)$ for $v \in N$. $C^+((N,q_T))$ is a free $R$-module of rank $4$; it is the quaternion ring $\Lambda_T$.  The map $v_1v_2 \cdots v_r \mapsto v_r \cdots v_2v_1$ on the tensor algebra descends to an involution on $\Lambda_T$.  This involution is the conjugation $v \mapsto v^*$ on the quaternion algebra.  The preferred choice of basis of $\Lambda_T$ is the element $1$, together with $v_1 = e_2e_3, v_2 = e_3e_1, v_3 = e_1e_2$.  It is immediate that this is a good basis, because it is essentially the definition of the Clifford algebra
\[\left(\begin{array}{c} e_1 \\ e_2 \\ e_3 \end{array}\right) \left(\begin{array}{ccc} e_1 & e_2 & e_3 \end{array}\right) = \left(\begin{array}{ccc} a & v_3 & v_2^* \\ v_3^* & b & v_1 \\ v_2 & v_1^* & c\end{array}\right)\]
that the matrix on the right factorizes as on the left, and thus (heuristically, at least) is rank one.

If $\Lambda, (v_1,v_2, v_3)$, $\Lambda', (v_1',v_2',v_3')$ are two quaternion rings over $R$ with good bases, we say a map $\phi: \Lambda \rightarrow \Lambda'$ is an isomorphism of good-based quaternion rings if $\phi$ is an isomorphism of $R$-algebras, preserving the involution, and $\phi(v_i) = v_i'$, $i = 1,2,3$.  The constructions associating a ternary quadratic form $q_T$ to a good-based quaternion ring and the even Clifford algebra to a ternary quadratic form are inverse to one another.  Furthermore, there is a notion of ``reduced discriminant'' of quaternion rings, see \cite[pg. 8]{gL}, and similarly for ternary quadratic forms.
\begin{proposition}[Gross-Lucianovic \cite{gL}]\label{gLBij} Every quaternion ring has a good basis.  The even Clifford construction induces a discriminant preserving bijection between ternary quadratic forms $q_T$ and isomorphism classes of good-based quaternion rings. \end{proposition}
They furthermore give a $\GL_3(R)$ action on ternary quadratic forms, and on good bases, and then descend the bijection in the proposition to a bijection between $\GL_3(R)$ orbits of ternary quadratic forms and isomorphism classes of quaternion rings over $R$.

Because the bijection of Proposition \ref{gLBij} preserves discriminant, one obtains the following corollary.
\begin{corollary}\label{4detT} Suppose the half-integral symmetric matrix $T$ corresponds to a maximal order in a quaternion algebra $B$ over $\Q$.  Denote by $D_B = \prod_{p \text{ ram}}{p}$ the product of the primes $p$ for which $B \otimes \Q_p$ is nonsplit.  Then $4 \det(T) = \pm D_B$. \end{corollary}

We now work over $\Z$.  If $\Lambda$ is a quaternion ring, and $x \in \Lambda$, we denote by $\iIm(x)$ the element of $\Lambda \otimes \Q$ with $x = \tr(x)/2 + \iIm(x)$. The following lemma is the same as the uniqueness of the $x_i$ so that $w_i - x_i$ form a good basis.
\begin{lemma}\label{ImPartsRk1} Suppose $B$ a quaternion algebra over $\Q$, $A_1, A_2 \in H_3(B)$.  Suppose furthermore that the off-diagonal elements of $A_1$ span the three-dimensional $\Q$ vector space $B/(\Q \cdot 1)$, and similarly for $A_2$.  If $A_1, A_2$ are rank one matrices with $\iIm(A_1) = \iIm(A_2)$, then $A_1 = A_2$. \end{lemma}

Below we will require the following definition.
\begin{definition}\label{A(T)defn} Suppose $T$ is a half-integral symmetric matrix, with entries as in (\ref{Tentries}). Let $\Lambda_{T}, (v_1,v_2,v_3)$ be the associated quaternion ring with good basis.  We define $A(T)$ to be the matrix
\begin{equation}\label{A(T)eqn} A(T) = \left(\begin{array}{ccc} a & v_3 & v_2^* \\ v_3^* & b & v_1 \\ v_2 & v_1^* & c\end{array}\right).\end{equation}
One has that $A(T)$ is rank one in the sense of Definition \ref{rnkH3B}. \end{definition}

For further applications, one wants to understand suborders of quaternion orders, together with their good bases.  Let $T$ be a half-integral symmetric matrix, and $\Lambda_T, (v_1,v_2,v_3)$ the associated quaternion ring with its good basis.  Suppose $m \in M_3(\Z) \cap \GL_3(\Q)$. Consider the three elements $w_1, w_2, w_3 \in \Lambda_T$ defined by the matrix equation
\begin{equation}\label{w_i}\left(\begin{array}{ccc} w_1 & w_2 & w_3 \end{array}\right) := \left(\begin{array}{ccc} v_1 & v_2 & v_3 \end{array}\right)m.\end{equation}
That is, $w_i = \sum_j{m_{ji}v_j}$.  Now consider the lattice $\Lambda_T(m) \subset \Lambda_T$ spanned by $1, w_1, w_2, w_3$.  It is a natural question to ask when the lattice $\Lambda_T(m)$ is closed under multiplication.  Lemma \ref{AITlemma} below answers this question.
 
\begin{definition} For $m \in \GL_3$, denote $c(m) = \det(m)\,^t m^{-1}$ the cofactor matrix of $m$. \end{definition}

\begin{lemma}\label{AITlemma} Denote by $J_T$ the set of $3 \times 3$ Hermitian matrices 
\[\left(\begin{array}{ccc} c_1 & x_3 & x_2^* \\ x_3^* & c_2 & x_1 \\ x_2 & x_1^* & c_3\end{array}\right)\]
with $c_i \in \Z$ and $x_i \in \Lambda_T$.  For $m \in \GL_3(\Q)$, denote by $v_i'$ the off-diagonal entries of $m^{-1}A(T)c(m)$, and $\Lambda'$ the lattice in $\Lambda \otimes \Q$ spanned by $1$ and $v_1',v_2',v_3'$.  Consider the following four statements:
\begin{enumerate}
\item $\Lambda_T(m)$ is closed under multiplication;
\item $m^{-1}Tc(m)$ is half-integral;
\item $m^{-1}A(T)c(m)=:A'$ is in $J_T$.
\item $\Lambda'$ is closed under multiplication.
\end{enumerate}
Then, if $m \in M_3(\Z)$, all four statements are equivalent.  If 3 holds, then $m \in M_3(\Z)$. In the case these equivalent conditions are satisfied, $1, v_1', v_2', v_3'$ is a good basis of $\Lambda_T(m)$. \end{lemma}
Actually, we prove the following stronger statement: $(3)$ implies $(2)$; $(2)$ and $(4)$ are equivalent.  Suppose $(1)$ holds, and denote by $A_0$ the rank one matrix associated to the good basis of $\Lambda_T(m)$ obtained by translating the $w_i$ by integers.  Then $A_0 = A'$, and thus $\Lambda_T(m) = \Lambda'$.  Hence $(1)$ implies $(2)$, equivalently $(4)$, and if $m \in M_3(\Z)$, then $(1)$ implies $(3)$.  If $m \in M_3(\Z)$, then $(4)$ implies $\Lambda_T(m) = \Lambda'$, and thus $(4)$ implies $(1)$.
\begin{proof}[Proof of Lemma \ref{AITlemma}] Denote by $V_3$ the defining three-dimensional representation of $\GL_3$.  One has the equivalence $\wedge^2 V_3^* \otimes \det \simeq V_3$.  Suppose $x_1, x_2, x_3$ are indeterminates.  In coordinates, this equivalence of representations of $\GL_3$ amounts to the identity
\begin{equation*} \det(m) m^{-1} \left(\begin{array}{ccc} & x_3 & -x_2 \\ -x_3 & & x_1 \\ x_2 & -x_1 & \end{array}\right)\,^tm^{-1} = \left(\begin{array}{ccc} & x_3' & -x_2' \\ -x_3' & & x_1' \\ x_2' & -x_1' & \end{array}\right), \end{equation*}
where $x_i' = \sum_j{m_{ji} x_j}$.  Thus since $\iIm(A(T))$ is antisymmetric,
\begin{equation}\label{Im(w_i)}\iIm(A') = m^{-1}\iIm(A(T))c(m) = \left(\begin{array}{ccc} 0 & \iIm(w_3) & -\iIm(w_2) \\ -\iIm(w_3) & 0 & \iIm(w_1) \\ \iIm(w_2) & - \iIm(w_1) & 0 \end{array}\right).\end{equation}

Now, immediately item 3 implies item 2, and 2 and 4 are equivalent because $m^{-1}Tc(m)$ exactly describes the multiplication rules for $\Lambda'$ in $\Lambda \otimes \Q$.  Suppose now that item 1 holds.  Then one may change $w_1, w_2, w_3$ by integers to make them into a good basis. Hence we have $A_0, A'$ both rank one, and with the same imaginary parts by (\ref{Im(w_i)}).  Thus by Lemma \ref{ImPartsRk1}, $A_0 = A'$.  Since by definition the off-diagonal entries of $A_0$ generate $\Lambda_T(m)$, and the off-diagonal entries of $A'$ generate $\Lambda'$, we get $\Lambda_T(m) = \Lambda'$.  Hence item 1 implies 4 or equivalently 2.  If $m \in M_3(\Z)$, then this $A_0$ is in $J_T$, and hence when $m$ is integral, item 1 implies item 3.

Now suppose item 4 holds, and $m$ is integral. We have that $\iIm(w_i) = \iIm(v_i')$. Since we assume item 4, the $v_i'$ have integral trace, and since $m$ is integral, the $w_i$ have integral trace.  Hence $w_i - v_i' \in \frac{1}{2}\Z$.  We claim the difference is in $\Z$.  Indeed, write $v_i' = w_i + \ell/2$, with $\ell$ an integer.  Then 
\[2N(v_i') = 2(w_i + \ell/2)(w_i^* + \ell/2) = 2N(w_i) + \ell \tr(w_i) + \ell^2/2.\]
Again since $\Lambda'$ is closed under multiplication and $m$ is integral, $v_i'$ and $w_i$ have integral norm.  Hence $\ell^2/2$ is an integer, and thus $\ell$ is even.  Consequently, if $m$ is integral, then item 4 implies $\Lambda' = \Lambda_T(m)$, and hence 4 implies 1.

Finally, suppose item 3 holds. Then by definition, $v_i'$ is in $\Lambda_T$.  On the other hand, we have $\iIm(v_i') = \sum_{j}{m_{ji}\iIm(v_j)}$, and thus $v_i' = a_i + \sum_{j}{m_{ji}v_j}$ for some $a_i$ in $\Q$.  Since $v_i'$ is in $\Lambda_T$ and $1, v_1, v_2, v_3$ are a basis on $\Lambda_T$, it follows that the $a_i$ and $m_{ji}$ are in $\Z$, and thus $m \in M_3(\Z)$, as desired.\end{proof}

\begin{remark} The equivalence of (1), (2) and (4) in Lemma \ref{AITlemma} can also be deduced easily from Proposition 22.4.12 of \cite{voightBook}.  We thank John Voight for pointing this out to the author. \end{remark}

\begin{remark} Lemma \ref{AITlemma} suggests the following partial order on non-degenerate ternary quadratic forms.  Suppose $q_1(x,y,z)$ and $q_2(x,y,z)$ are integral ternary quadratic forms, with associated half-integral symmetric matrices $T_1$ and $T_2$, respectively.  We say that $q_i$ is non-degenerate to mean $\det(T_i) \neq 0$.  Define a partial order on the set of non-degenerate integral ternary quadratic forms by $q_1 \leq q_2$ if there exists $m \in M_3(\Z) \cap \GL_3(\Q)$ such that $T_1 = m^{-1}T_2 c(m)$.  Suppose $\Lambda_i$ is the even Clifford algebra associated to $q_i(x,y,z)$.  Then Lemma \ref{AITlemma} implies that there is an injective ring map $\Lambda_1 \rightarrow \Lambda_2$ if and only if $q_1 \leq q_2$.\end{remark}
\section{A Dirichlet Series for the Spin $L$-function}\label{sec:dirichlet}
In this section we combine Lemma \ref{AITlemma} with a result of Evdokimov \cite[Theorem 3]{evdokimov} to give a Dirichlet Series for the Spin $L$-function of level one Siegel modular forms on $\GSp_6$.  Most of the work has already been done in \cite{evdokimov}.  Namely, \cite[Theorem 3]{evdokimov} very nearly gives a Dirichlet series for this $L$-function, except this result of Evdokimov involves an ``error term'' that must be removed on the way to proving Theorem \ref{MainThm}\footnote{The result \cite[Theorem 4]{evdokimov} is a Dirichlet series for this $L$-function, but the hypotheses of this theorem of loc. cit. are too stringent to use in the proof of Theorem \ref{MainThm} for general quaternion algebras $B$ ramified at infinity; see Remark \ref{evDS}.}. However, it turns out that this error term can be trivialized by applying the arithmetic invariant theory of ternary quadratic forms, in the form of Lemma \ref{AITlemma}, together with the fact that $B_0$ is a maximal order.  We first recall precisely \cite[Theorem 3]{evdokimov}, and then apply Lemma \ref{AITlemma} to obtain Corollary \ref{dirSeriesAIT}, which gives the desired Dirichlet series.

To state \cite[Theorem 3]{evdokimov}, we need to introduce some notation, largely following \cite{evdokimov}. Denote by $P_6$ the Siegel parabolic of $\GSp_6$, i.e., the set of elements $\mm{A}{B}{C}{D}$ in $\GSp_6$ with $C=0$.  Set $S_p = \{M \in \GSp_6(\Z[p^{-1}]): \nu(M) \in p^\Z\}$, $\Gamma_0 = \Sp_6(\Z) \cap P_6(\Q)$, and $S_{0,p} = S_p \cap P_6(\Q)$.  Denote by $L_{0,p}$ the double coset ring of $\Gamma_0$ in $S_{0,p}$. 

We define some elements of $L_{0,p}$.  Set
\[\Pi_0(p) = \Gamma_0\diag(p,p,p,1,1,1)\Gamma_0 \text{ and } \Pi_1(p) =\Gamma_0 \diag(p,p,1,1,1,p)\Gamma_0.\]
Define
\[\Pi_{2,0}^0(p) = \Gamma_0 \diag(p^2,p,p,1,p,p)\Gamma_0 \text{ and } \Pi_{3,0}^0(p) = \Gamma_0 \diag(p,p,p,p,p,p)\Gamma_0\]
For integers $0 \leq r \leq i$, set $S_i^r$ to be the set of $A \in M_i(\Z)$ with $\,^t A = A$, and $rank_{\F_p}A = r$.  Define an equivalence relation on $S_i^r$ by $A \sim A'$ if there exists $U \in \GL_i(\Z)$ such that $A' \equiv \,^tU A U$ modulo $p$.  Now, define
\[ \Pi_{2,0}^1(p) = \sum_{A \in S_2^1 \slash \sim}{\Gamma_0 \left(\begin{array}{ccc|ccc} p^2 & & & & & \\  & p& & & a_{22} & a_{23}\\ & &p & &a_{23} &a_{33} \\ \hline  & & & 1& & \\  & & & & p& \\  & & & & & p\end{array}\right) \Gamma_0}\]
where $A = \mm{a_{22}}{a_{23}}{a_{23}}{a_{33}}$.  Similarly, define
\[ \Pi_{3,0}^1(p) = \sum_{A \in S_3^1 \slash \sim}{\Gamma_0 \left(\begin{array}{ccc|ccc} p & & & a_{11}& a_{12}& a_{13}\\  & p& & a_{12}& a_{22} & a_{23}\\ & &p &a_{13} &a_{23} &a_{33} \\ \hline  & & & p& & \\  & & & & p& \\  & & & & & p\end{array}\right) \Gamma_0}\]
where $A = (a_{ij})$.  Set $\Pi_1^2(p) = p\left(\Pi_{2,0}^1(p) + \Pi_{2,0}^0(p)\right)$ and $\Pi_{1}^3(p) = p^3\left(\Pi_{3,0}^1(p) + \Pi_{3,0}^0(p)\right)\Pi_0(p)$.

Denote by $\mathcal{F}_k$ the space of holomorphic functions $f: \mathcal{H}_3 \rightarrow \C$ that satisfy $f((AZ + B) D^{-1}) = \det(D)^k f(Z)$ for all $\mm{A}{B}{}{D} \in P_6(\Q) \cap \Sp_6(\Z)$.  The ring $L_{0,p}$ acts on $\mathcal{F}_k$ as follows.  Suppose $\Gamma_0 h \Gamma_0 \subseteq S_{0,p}$ is a double coset.  Then $\Gamma_0 h \Gamma_0$ has a single coset decomposition
\[\Gamma_0 h \Gamma_0 = \bigsqcup_i \Gamma_0 \left(\begin{array}{cc} A_i & B_i \\ & D_i \end{array}\right),\]
$\nu \left(\mm{A_i}{B_i}{}{D_i}\right) = p^{\delta_i}$.  One defines
\[ \left(f|\Gamma_0 h \Gamma_0\right)(Z) = \sum_i{ p^{(3k-6)\delta_i}(\det(D_i))^{-k}f((A_iZ+B_i)D_i^{-1})}.\]

Suppose $\phi$ is as in Definition \ref{holAssumption}, and $f_\phi(Z) = \sum_{T' > 0}{a(T') e^{2\pi i \tr(T'Z)}}$ is the associated level one Siegel modular form.  Let us write $L_{classical}(\pi,Spin,s) = L(\pi,Spin,s-3r+3)$ for the classically normalized Spin $L$-function of $\pi$. Define $Q_p^1(z) = 1 - \Pi_0(p) z$ and $Q_p^2(z) = 1 - \Pi_1(p) z + \Pi_1^2(p) z^2 - \Pi_1^3(p) z^3$, elements of $L_{0,p}[z]$, the polynomial ring over $L_{0,p}$ in the indeterminate $z$.  For $m \in \GL_3(\Q)$, define 
\[\Xi_T(m) = \begin{cases} 1 &\mbox{if } m^{-1}T c(m) \in M_3(\Z) \\ 0 &\text{otherwise} \end{cases}.\]
The following theorem is essentially \cite[Theorem 3]{evdokimov}.
\begin{theorem}[Evdokimov]\label{evdThm3} Let the notations be as above.  Set $M_3^+(\Z)$ to be the elements of $M_3(\Z)$ with positive determinant. Then
\begin{equation}\label{evdExp}\sum_{\lambda \geq 1, m \in M_3^+(\Z) \slash \SL_3(\Z)}{\frac{\Xi_T(m)}{\lambda^s \det(m)^{s-2r+3}} a(\lambda m^{-1}Tc(m))} = C_{f_\phi}(s,T) \frac{L_{classical}(\pi,Spin,s)}{\zeta_{D_B}(2s-6r+8)\zeta(2s-6r+6)},\end{equation}
where
\[ \zeta_{D_B}(s) = \prod_{p < \infty, p \nmid D_B}{\left(1 - p^{-s}\right)^{-1}}\] 
and $C_{f_\phi}(s,T)$ is the $T$'th Fourier coefficient of
\[f_\phi \left| \prod_{p| 4 \det(T) }{Q_p^1(p^{-s})Q_p^2(p^{-s})}.\right.\]
\end{theorem}
The function $C_{f_\phi}(s,T)$ in Theorem \ref{evdThm3} is the ``error term'' that we eluded to above.  By applying Lemma \ref{AITlemma}, we will see that the hypothesis that $T$ corresponds to a maximal order implies that $C_{f_\phi}(s,T) = a(T)$ is independent of $s$.  Hence under this assumption, the left-hand side of (\ref{evdExp}) yields a Dirichlet series for the Spin $L$-function.  

If $\mathcal{O}$ is an order in some quaternion algebra over $\Q$ ramified at infinity, then $\mathcal{O} = \Lambda_{T'}$ is the even Clifford algebra associated to some half-integral, symmetric, positive-definite matrix $T'$.  We write $a(\mathcal{O}) := a(T')$.  Since $a(T') = a(k^{-1} T' c(k))$ for $k \in \SL_3(\Z)$, the notation $a(\mathcal{O})$ is independent of the choice of $T'$, and thus well-defined. We have the following corollary of Lemma \ref{AITlemma} and Theorem \ref{evdThm3}.
\begin{corollary}\label{dirSeriesAIT} Suppose $\pi$ is as in Definition \ref{holAssumption}, $f_\phi(Z) = \sum_{T' > 0}{a(T') e^{2\pi i \tr(T'Z)}}$ is the Fourier expansion of the associated Siegel modular form, and the even Clifford algebra $B_0$ associated to $T$ is a maximal order in $B_0 \otimes \Q$.  Then
\begin{equation}\label{dSAIT} \sum_{\lambda \geq 1, \mathcal{O} \subseteq B_0} \frac{a(\Z + \lambda \mathcal{O})}{\lambda^s [B_0: \mathcal{O}]^{s-2r+3}} = a(B_0) \frac{L_{classical}(\pi,Spin,s)}{\zeta_{D_B}(2s-6r+8)\zeta(2s-6r+6)}, \end{equation} 
where $[B_0: \mathcal{O}]$ denotes the index of $\mathcal{O}$ in $B_0.$ \end{corollary}
\begin{proof} Applying Lemma \ref{AITlemma}, the left-hand side of (\ref{dSAIT}) is seen to be equal to the left-hand side of (\ref{evdExp}).  

The key point of the corollary is that, as mentioned above, $B_0$ being a maximal order implies $C_{f_\phi}(s,T) = a(T) = a(B_0)$.  To see this equality, consider the action of the operators $\Pi_0(p),$ $\Pi_1(p)$, $\Pi_1^2(p)$ and $\Pi_1^3(p)$ on the Fourier coefficients of $f_\phi$.  Suppose $\Pi(p)$ is one of these operators.  For a quaternion order $\mathcal{O}$ and a Siegel modular form $f$, set $a(\mathcal{O}, f|\Pi(p))$ the $\mathcal{O}$ Fourier coefficient of $f|\Pi(p)$.  Applying Lemma \ref{AITlemma}, one verifies that $a(\mathcal{O}, f|\Pi(p))$ is a linear combination of $a(\mathcal{O}', f)$ for certain orders $\mathcal{O}' \supseteq \mathcal{O}$ with $[\mathcal{O}':\mathcal{O}] = p, p^2$ or $p^3$.  Since $B_0$ is a maximal order, $a(B_0, f|\Pi(p)) = 0$, and the equality $C_{f_\phi}(s,T) =a(B_0)$ follows. \end{proof}

\begin{remark}\label{evDS} In \cite[Theorem 4]{evdokimov}, Evdokimov also uses his result \cite[Theorem 3]{evdokimov} to obtain a Dirichlet series for the Spin $L$-function, but only under the assumption that $f_\phi$ satisfies $a(T') = 0$ for all $T'$ with $4\det(T')$ strictly dividing $4\det(T)$.\end{remark}
\section{The Rankin-Selberg integral}\label{sec:Global}
In this section we define the map $\GSp_6 \rightarrow \GG$ and the Eisenstein series on $\GG$.  We also give the global Rankin-Selberg convolution and unfold it.
\subsection{Construction of global integral}\label{gsp6toG}
We first embed $\GSp_6$ into $\GG$, preserving the similitude.  Denote by $W_6$ the defining six-dimensional representation of $\GSp_6$, and $\langle \; , \; \rangle$ the invariant symplectic form.  We let $\GSp_6$ act on $W_6$ on the right.  In matrices, $\GSp_6$ is the set of $g \in \GL_6$ satisfying $g J_6 \,^t g = \nu(g) J_6$, where $J_6 = \mm{}{1_3}{-1_3}{}$ and the similitude $\nu(g) \in \GL_1$.  The action of $\GSp_6$ on $W_6$ is then the usual right action of matrices on row vectors.

Pick a symplectic basis $e_1, e_2, e_3, f_1, f_2, f_3$ of $W_6$; that is, $\langle e_i , f_j \rangle = \delta_{ij}$.  We fix a quaternion algebra $B$ over $\Q$ that is ramified at infinity.  For such a $B$, denote by $W = \Q \oplus H_3(B) \oplus H_3(B) \oplus \Q$ Freudenthal's space, as considered in section \ref{sec:Freud}.  Now consider the linear map $W \rightarrow (\bigwedge^3 W_6\otimes \nu^{-1}) \otimes_\Q B$ defined as follows:  Set $e_i^* = e_{i+1} \wedge e_{i+2}$ and $f_i^* = f_{i+1} \wedge f_{i+2}$ with the indices taken modulo $3$.  Then we map $(a,b,c,d)$ in $W$ to
\[ a e_1 \wedge e_2 \wedge e_3 + \left(\sum_{i,j}{b_{ij}e_i^* \wedge f_j}\right) +\left(\sum_{i,j}{ c_{ij}f_i^* \wedge e_j}\right) + d f_1 \wedge f_2 \wedge f_3.\]
With this definition, and the natural action of $\GSp_6$ on $\bigwedge^3 W_6 \otimes \nu^{-1}$, $\GSp_6$ preserves the image of $W$, and preserves the symplectic and quartic form on $W$, up to similitude.  These facts are checked in the following lemma.
\begin{lemma} Identify $W$ with its image in $\left(\bigwedge^3 W_6 \otimes \nu^{-1}\right) \otimes_\Q B$.  Then, the element $\mm{1}{X}{}{1}$ of $\GSp_6$ acts on $W$ as $n(X)$, and the element $\mm{0}{1}{-1}{0}$ acts on $W$ via $(a,b,c,d) \mapsto (-d,c,-b,a)$.  If $\mm{m}{}{}{n} \in \GSp_6$ has similitude $\nu$, then
\[(a,b,c,d)\mm{m}{}{}{n} = \nu^{-1}(\det(m)a, \det(m) m^{-1}b n, \det(n)n^{-1}cm,\det(n)d).\]
It follows that the inclusion $W \rightarrow \left(\bigwedge^3 W_6 \otimes \nu^{-1}\right) \otimes_\Q B$ induces an inclusion $\GSp_6 \rightarrow \GG$, preserving similitudes. \end{lemma}
\begin{proof} The formula for the action of $\mm{}{1}{-1}{}$ is immediate.  We check some of the other formulas.

Consider the action of $\mm{m}{}{}{n}$.  One first computes that under the action of $m$, $e_i^* \mapsto \sum_j{c(m)_{ij} e_j^*}$,  where $c(m)= \det(m)\,^t m^{-1}$ is the cofactor matrix of $m$.  Ignoring the twist by the similitude $\nu^{-1}$ in $\wedge^3 W_6 \otimes \nu^{-1}$, one then obtains
\begin{align*} \sum_{i,j}{b_{ij} e_i^* \wedge f_j} & \mapsto \sum_{i, j, t, s}{b_{ij} c(m)_{it} n_{js} e_t^* \wedge f_s} \\ & = \sum_{t, s}{(\,^tc(m) b n)_{t s} e_t^* \wedge f_s}. \end{align*}
This proves the formula for the action of $\mm{m}{}{}{n}$ on the $b$'s of $(a,b,c,d)$.  The action on the $c$'s is similar, and the action on the $a$'s and $b$'s are trivial.

For $X$ in $H_3(B)$, denote by $L(X)$ the logarithm, inside $\mathrm{End}(W)$, of the element $n(X)$ of $\GL(W)$.  That is, set $L(X) = \sum_{m \geq 1}{(-1)^{m-1}\frac{(n(X)-1)^{m}}{m}}.$  From the definition (\ref{n(X)form}) of $n(X)$, one sees $(n(X) - 1)^4 = 0$, and thus the sum defining $L(X)$ terminates after $4$ terms.  One gets
\[(a,b,c,d)L(X) = (0, a X, b \times X, \tr(c,X))\]
and $n(X) = 1 + L(X) + L(X)^2/2 + L(X)^3/6.$  

Similarly, for $\mm{1}{X}{}{1}$ in $\GSp_6$, denote by $n_6(X)$ its action on $W_6$, $n'(X)$ its action on $\left(\wedge^3 W_6 \otimes \nu^{-1}\right) \otimes B$, and $L_6(X)$, respectively $L'(X)$, the logarithms of these actions.  Since $n_6(X) = 1 + L_6(X)$, and the action of $n(X)$ on $\wedge^3 W_6$ is defined by cubic polynomials in the entries of $\mm{1}{X}{}{1}$, one obtains $n'(X) = 1 + L'(X) + L'(X)^2/2 + L'(X)^3/6.$  Hence to check that $n'(X)$ restricts to $n(X)$ on $W$, it suffices to check the analogous statement for $L'(X)$ and $L(X)$.

So, suppose $X$ is a $3 \times 3$ symmetric matrix.  We consider the action $(0,b,0,0)L'(X)$.  One has
\[\left( e_i^* \wedge f_j\right) L'(X) = \sum_k{X_{i+1, k} f_k \wedge e_{i+2} \wedge f_j + X_{i+2, k} e_{i+1} \wedge f_k \wedge f_j},\]
where $X_{a,b}$ denotes the $a,b$ matrix entry of $X$.  Hence
\begin{align}\label{bL'(X)} \left(\sum_{ij}{b_{ij} e_i^* \wedge f_j} \right) L'(X) &= \sum_{j, k, \ell}{\left( b_{\ell + 1, j} X_{\ell -1, k} - X_{\ell+1, k} b_{\ell -1, j}\right) f_j \wedge f_k \wedge e_\ell} \\ &= \sum_{\alpha, \beta}{V_{\alpha, \beta} f_{\alpha}^* \wedge e_\beta} \nonumber \end{align}
for some element $V$ of $M_3(B)$.  For a general element $Y$ of $M_3(B)$, denote by $c_i(Y), a_i(Y), a_i(Y)'$ the coordinates of $Y$:
\[Y = \left(\begin{array}{ccc} c_1(Y) & a_3(Y) & a_2(Y)' \\ a_3(Y)' & c_2(Y) & a_1(Y) \\ a_2(Y) & a_1(Y)' & c_3(Y) \end{array}\right).\]
Then (\ref{bL'(X)}) yields
\begin{align}\label{Vcoords} c_t(V) &= c_{t-1}(X)c_{t+1}(b) + c_{t+1}(X)c_{t-1}(b) - (a_t(X),a_t(b)), \nonumber\\ 
a_t(V) &= a_{t+1}'(b) a_{t-1}'(X) + a_{t+1}'(X) a_{t-1}'(b) - \left(c_t(X) a_t(b) + c_t(b) a_t(X)\right), \\ a_t'(V) &= a_{t+1}(b) a_{t-1}(X) + a_{t+1}(X) a_{t-1}(b) - \left(c_t(X) a_t'(b) + c_t(b) a_t'(X)\right). \nonumber \end{align}
The indices in (\ref{Vcoords}) are taken modulo $3$.  Linearizing the equation (\ref{hSharp}) for the map $\#$ yields exactly the formulas (\ref{Vcoords}), and thus $V = b \times X$.  Hence $(0,b,0,0)L'(X) = (0,0, b\times X, 0)$.

The formulas $(a,0,0,0) L'(X) = (0,a X, 0,0)$ and $(0,0,c,d)L'(X) = (0,0,0,\tr(c,X))$ are much easier to establish.  Hence $L'(X)$ preserves $W$, and restricts to the action of $L(X)$.  This completes the proof of the lemma. \end{proof}

We need to define an integral structure on $W$.  To do this, first we fix a maximal order $B_0$ in $B$.  Then
\begin{equation}\label{W(Z)} W(\Z) := W(\Q) \cap \left\{\left(\wedge^3 W_6(\Z) \otimes \nu^{-1}\right) \otimes_\Z B_0\right\},\end{equation} and for a finite prime $p$, $W(\Z_p) := W(\Z)\otimes \Z_p$.  Define $J_0$ to be the lattice in $H_3(B)$ consisting of those
\begin{equation}\label{J_0}\left(\begin{array}{ccc} c_1 & x_3 & x_2^* \\ x_3^* & c_2 & x_1 \\ x_2 & x_1^* & c_3\end{array}\right)\end{equation}
with $c_i \in \Z$ and $x_i \in B_0$.  Then $(a,b,c,d) \in W(\Q)$ is in $W(\Z)$ if and only if $a, d \in \Z$ and $b, c \in J_0$.  It is clear from (\ref{W(Z)}) that $\GSp_6(\Z_p)$ stabilizes $W(\Z_p)$ for all finite primes $p$.

We now define the Eisenstein series on $\GG$.  Recall the rank one element $f = (0,0,0,1)$ of $W$.  Denote by $P$ the parabolic subgroup of $\GG$ stabilizing the line $\Q f$.  The Eisenstein series will be associated to a character of $P$.  To define this Eisenstein series, we select a particular Schwartz-Bruhat function $\Phi = \prod_{v}\Phi_v$ on $W(\A)$.  At finite primes, $\Phi_p$ is the characteristic function of $W(\Z_p)$.  In section \ref{arch} will define a Schwartz function $\Phi_{\infty,2r}$ on $W(\R)$ that is appropriate for cusp forms $\phi$ that correspond to a Siegel modular form of weight $2r$.  The Schwartz function $\Phi_{\infty,2r}$ will be a Gaussian times a polynomial of degree $2r$.  We set
\[f^\Phi(g,s) = |\nu(g)|^s \int_{\GL_1(\A)}{\Phi(tfg)|t|^{2s}\,dt},\]
and 
\[E^\Phi(g,s) = \sum_{\gamma \in P(\Q) \backslash \GG(\Q)}{f^\Phi(\gamma g,s)}.\]
This is not yet the normalized Eisenstein series.  Assume $\Phi_{\infty,2r}$ is chosen as in section \ref{arch}, and denote by $D_B = \prod_{p < \infty, \text{ ramified}}{p}$ the product of the finite primes ramified in $B$.  Then the normalized Eisenstein series is
\begin{align*} E_{2r}^*(g,s) =& \pi^{-(s+r)}D_B^s \Gamma_\R(2s+2r-4)\Gamma_\R(2s+2r-8) \left(\prod_{p < \infty \text{ split}}{\zeta_p(2s-2)\zeta_p(2s-4)}\right)\\ &\times \left(\prod_{p < \infty \text{ ramified}}{\zeta_p(2s-4)} \right) E^\Phi(g,s). \end{align*}
Here, as usual, $\Gamma_\R(s) = \pi^{-s/2}\Gamma(s/2)$.  The following theorem is proved in section \ref{Eisenstein}.
\begin{theorem}\label{NormEis} The Eisenstein series $E_{2r}^*(g,s)$ has finitely many poles, all contained in the set $\frac{1}{2} \Z \cap [0,5]$, and satisfies the functional equation $E_{2r}^*(g,s) = E_{2r}^*(g,5-s)$. \end{theorem}
For a cusp form $\phi$ associated to a level one, holomorphic Siegel modular form of weight $2r$, the global integral is
\[I_{2r}^*(\phi,s) = \int_{\GSp_6(\Q)Z(\A) \backslash \GSp_6(\A)}{\phi(g)E_{2r}^*(g,s) \,dg}.\]
By Theorem \ref{NormEis}, $I_{2r}^*(\phi,s) = I_{2r}^*(\phi,5-s)$.
\begin{remark} This is the first place where defining the group $\GG$ in terms of the Freudenthal construction gives one a substantial technical advantage.  Namely, because we have defined $\GG$ in terms of the representation $W$, we were able to construct the Eisenstein section $f^\Phi$ in terms of a characteristic function on $W$.  We will see in section \ref{unram} that this definition plays a significant role in easing the computation. \end{remark}
\subsection{Unfolding}\label{Unfolding}
To unfold this integral, we rely on the following lemma.
\begin{lemma}\label{d=1}Every rank one element of $W$ has a $\GSp_6(\Q)$ translate with $d =1$. \end{lemma}
\begin{proof}  Suppose our rank one element $v$ is $(a,b,c,d)$.  If either $a$ or $d$ is nonzero, we are done.  So, suppose $v = (0,b,c,0)$.  Then $(0,b,c,0)n(X) = (0,b,c+b \times X,\tr(c,X) + \tr(b,X^\sharp))$.  Since $v$ is rank one, by Proposition \ref{rk1Sharp}, $b^\sharp = c^\sharp = 0$.  From the formula (\ref{hSharp}) for $\sharp$, since $B$ is not split, we cannot have both $c$ purely imaginary and $c^\sharp = 0$, unless $c=0$.  The same for $b$.  Without loss of generality, assume $c \neq 0$.  Since $c$ is not purely imaginary, there exists $X \in H_3(\Q)$ with $\tr(c,X) \neq 0$.  Hence $\lambda \tr(c,X) + \lambda^2\tr(b,X^\sharp)$ is not zero for some $\lambda$ in $\Q$, proving the lemma. \end{proof}

To do the unfolding, we need to understand the double coset space $P(\Q) \backslash \GG(\Q) \slash \GSp_6(\Q)$, or equivalently, the $\GSp_6(\Q)$ orbits on the rank one lines.  Recall that $T$ denotes a half-integral symmetric $3 \times 3$ matrix that corresponds to the order $B_0$, and set $A(T)$ to be the rank one matrix in $H_3(B)$ from Lemma \ref{AITlemma}.  Define $f_\O = (0,0,A(T),0)$ in $W$.  By Proposition \ref{rk1Sharp}, $f_\O$ is rank one.  If $v$ is a rank one element of $W(\Q)$, an orbit $\Q v \GSp_6(\Q)$ is said to be \emph{negligible} if the stabilizer of $\Q v$ in $\GSp_6(\Q)$ contains the unipotent radical of a parabolic subgroup of $\GSp_6$.
\begin{proposition} \label{orbit calc} The line $\Q f_\O$ represents the unique open orbit in $P(\Q) \backslash \GG(\Q) \slash \GSp_6(\Q)$.  There are an infinite number of other orbits, but they are all negligible. The stabilizer of the line $\Q f_\O$ is the set of elements of $\GSp_6$ of the form $g = \mm{\lambda_1}{}{}{\lambda_2} \mm{1}{U}{}{1}$ with $\lambda_1, \lambda_2$ in $\GL_1(\Q)$ and $\tr(TU) = 0$.\end{proposition}
\begin{remark} This is another spot where defining the group $\GG$ in terms of the Freudenthal construction gives one a serious technical advantage.  Here, since $P(\Q) \backslash \GG(\Q)$ parametrizes the rank one lines in $W(\Q)$, the use of $W$ allows one to see the action of $\GSp_6(\Q)$ on the coset space $P(\Q) \backslash \GG(\Q)$ as a \emph{linear} action.  This linearization makes the double coset computation simple. \end{remark}
Before giving the proof, define $J_6 = e_1 \wedge f_1 + e_2 \wedge f_2 + e_3 \wedge f_3$. Note that if $a_1, a_2, a_3$ in $B$ are purely imaginary, and $c_1, c_2, c_3$ are in $\Q$, then 
\begin{equation}\label{c matrix}(a_1f_1 + a_2 f_2 + a_3f_3) \wedge J_6 + c_1e_1 \wedge f_2 \wedge f_3 + c_2f_1 \wedge e_2 \wedge f_3 + c_3f_1 \wedge f_2 \wedge e_3\end{equation}
is the element
\[(0,0,\left(\begin{array}{ccc} c_1 & a_3 & -a_2 \\ -a_3 & c_2 & a_1 \\ a_2 & -a_1 & c_3 \end{array}\right),0)\]
of $W$.
\begin{proof}[Proof of Proposition \ref{orbit calc}.] Suppose we are given a rank one element $v = (a,b,c,d)$ of $W(\Q)$.  We will first use the action of $\GSp_6(\Q)$ to put it into a standard form.  By Lemma \ref{d=1}, we may assume $d = 1$.  Then our element is of the form $(N(c),c^\#,c,1)$.  Now use the transformations $\overline{n}(Y)$ to make $c$ purely imaginary.

Denote by $M_{P,6}$ the Levi of the Siegel parabolic on $\GSp_6$, i.e., the elements of $\GSp_6$ of the form $\mm{m}{}{}{n}$.  By formula (\ref{c matrix}), $M_{P,6}$ acts on the space of such $c$'s as $B^{\tr=0} \otimes_\Q span\{ f_1, f_2, f_3 \}$.  That is, $M_{P,6}$ acts on these $c$'s via its usual action on $span\{f_1, f_2, f_3\} \subseteq W_6$, linearly extended to $B^{\tr=0} \otimes_\Q span\{ f_1, f_2, f_3 \}$.  Pick a basis $i, j, k$ of $B^{\tr = 0}$, and suppose $c = i \otimes v_1 + j \otimes v_2 + k \otimes v_3$. We have four different cases, depending on whether $dim\, span \{ v_1, v_2, v_3 \}$ is 0,1,2 or 3.  (Note that this dimension is not a $\GSp_6$ invariant of the orbit, just an $M_{P,6}$ invariant.) If the dimension is 0, i.e, $c = 0$, then we are in the orbit of $f$, which is negligible.  There are an infinite number of orbits with $dim\, span \{ v_1, v_2, v_3 \} = 1$. Given such an orbit, we can use $M_{P,6}$ to move $c$ to $h \otimes f_1$ with $h \in B^{\tr=0}$.  Then this orbit is stabilized by $U_Q(f_1)$, the unipotent radical of the parabolic subgroup of $\GSp_6$ stabilizing the line $\Q f_1 \subseteq W_6$.  Thus, this orbit is also negligible. We will show that there is a unique orbit with $dim=3$, and that this orbit is represented by $f_\mathcal{O}$.  Furthermore, we will see that the orbit corresponding to $dim = 2$ is the same as the $dim = 3$ orbit.

First, for $dim = 3$, use $M_{P,6}$ to move $c$ to $i \otimes f_1 + j\otimes f_2 + k \otimes f_3$, proving that there is one such orbit.  

Now we consider a $dim = 2$ orbit.  Here we have $c = i \otimes v_1 + j \otimes v_2 + k \otimes v_3$ with one of the $v_i$ in the span of the other two.  Thus we may use $M_{P,6}$ to move this orbit to one with $c = h_1 \otimes f_1 + h_2 \otimes f_2$, with $h_1, h_2$ in $B^{\tr=0}$.  If $dim\, span \{ h_1, h_2 \} = 1$, then we are really in the $dim =1$ case.  Thus we have that $h_1, h_2$ are linearly independent.

Now 
\[c = \left(\begin{array}{ccc} 0 &0  & -h_2 \\ 0 & 0 & h_1 \\ h_2 & -h_1 & 0 \end{array}\right),\]
from which we obtain 
\[c^\sharp = \left(\begin{array}{ccc} -n(h_1) & h_2h_1 & 0 \\ h_1h_2 & -n(h_2) & 0 \\ 0 & 0 & 0 \end{array}\right)\]
and $N(c) = 0$.  Hence we have that $v = (N(c), c^\#, c, 1)$ is
\[-n(h_1)e_1^* \wedge f_1 - n(h_2) e_2^* \wedge f_2 + h_1h_2 e_2^* \wedge f_1 + h_2h_1 e_1^* \wedge f_2 + (h_1f_1 + h_2 f_2) \wedge J_6 + f_1\wedge f_2 \wedge f_3.\]
Now apply to $v$ the element of $\GSp_6$ that takes $e_3 \mapsto f_3$, $f_3 \mapsto -e_3$ and is the identity on $e_1,e_2,f_1,f_2$.  One obtains the element $v' = (0,0,h,0)$, with
\[h = \left(\begin{array}{ccc} -n(h_2) & -h_2h_1 & -h_2 \\ -h_1h_2 & -n(h_1) & h_1 \\ h_2 & -h_1 & -1 \end{array}\right).\]
Since $B$ is not split, $h_1 h_2$ is not in the span of $1, h_1, h_2$.  Take $h_3$ in $B^{\tr=0}$ with $h_1h_2 -h_3$ in $\Q$.  Then $\iIm(h) = (h_1 f_1 + h_2 f_2 + h_3 f_3) \wedge J_6$.  Since 
\[h_1 \otimes f_1 + h_2 \otimes f_2 + h_3 \otimes f_3 = i \otimes v_1 + j \otimes v_2 + k \otimes v_3\]
for some $v_1, v_2, v_3$ spanning $\langle f_1, f_2, f_3 \rangle$, the $dim =2$ orbit is the same as the $dim = 3$ orbit.

The element $f_\mathcal{O}$ clearly represents the $dim = 3$ orbit.  To compute the stabilizer, note that by (\ref{c matrix}) we immediately see that if $f_\mathcal{O}g = f_\mathcal{O}$, then $f_1g = f_1, f_2g = f_2, f_3g = f_3$.  Hence an element of the stabilizer must be of the form $g = \mm{\lambda}{}{}{1} \mm{1}{U}{}{1}$.  Applying this element to $f_\mathcal{O}$, we find it stabilizes $f_{\mathcal O}$ if and only if $\tr(TU) = 0$, completing the proof.
\end{proof}

Recall the partially normalized Eisenstein series $E^\Phi(g,s)$ from above, and set
\[I^\Phi(\phi,s) = \int_{\GSp_6(\Q) Z(\A) \backslash \GSp_6(\A)}{\phi(g)E^\Phi(g,s)\,dg}.\]
We will consider the integral $I^\Phi(\phi,s)$ for now, and add back in the other normalizing factors later. 

Write $U_P$ for the unipotent radical of the Siegel parabolic of $\GSp_6$.  The group $U_P$ consists of the matrices $\mm{1}{u}{}{1}$ with $u = \,^tu$.  Define $U_0 \subseteq U_P$ to be the elements with $\tr(Tu) = 0$.  Write $\psi: \Q \backslash \A \rightarrow \C^\times$ for the usual additive character of conductor 1 satisfying $\psi_\infty(x) = e^{2\pi i x}$ for $x \in \R$.  Define $\chi: U_P(\Q) \backslash U_P(\A) \rightarrow \C^{\times}$ via $\chi(\mm{1}{h}{}{1}) = \psi(\tr(Th)),$ and set 
\begin{equation}\label{phichi1}\phi_\chi(g) = \int_{U_P(\Q) \backslash U_P(\A)}{\chi^{-1}(u)\phi(ug) \, du}.\end{equation}
\begin{theorem}\label{thm:unfold} The global integrals $I^\Phi(\phi,s), I_{2r}^*(\phi,s)$ converge absolutely for all $s$ for which the Eisenstein series $E^\Phi(g,s)$, respectively, $E_{2r}^*(g,s)$ is finite, and thus define a meromorphic function of $s$.  For $Re(s)$ sufficiently large, the integral $I^\Phi(\phi,s)$ unfolds as
\begin{align*} I^\Phi(\phi,s) &= \int_{U_0(\A)\backslash \GSp_6(\A)}{\phi_\chi(g) |\nu(g)|^s \Phi(f_\O g)\,dg}\\ &= \int_{U_P(\A)\backslash \GSp_6(\A)}{\phi_\chi(g) |\nu(g)|^s \left(\int_{U_0(\A)\backslash U_P(\A)}{\chi(u)\Phi(f_\O u g)\,du}\right)\,dg}. \end{align*} \end{theorem}
\begin{proof} The Eisenstein series $E^\Phi(g,s)$ is a function of moderate growth wherever it is finite.  Since the cusp form $\phi$ decreases rapidly on Siegel sets, the integral $I^\Phi(\phi,s)$ converges absolutely for all $s$ for which the Eisenstein series is finite.

Suppose now $Re(s) >> 0$, and denote by $\mathcal{F}$ a fundamental domain for $\GSp_6(\Q)Z(\A)\backslash \GSp_6(\A)$.  Similar to the above, but easier, $\sum_{\gamma \in P(\Q)\backslash \GG(\Q)}{|f^\Phi(\gamma g,s)|}$ converges to a function of moderate growth since $Re(s) >>0$.  Then again since $\phi$ decreases rapidly on Siegel sets, the sum
\[\sum_{\gamma \in P(\Q)\backslash \GG(\Q)}{\int_{\mathcal{F}}{|f^\Phi(\gamma g,s)||\phi(g)|\,dg}}\]
is finite.  Hence even though there are an infinite number of $\GSp_6(\Q)$ orbits on $P(\Q)\backslash \GG(\Q)$, we may analyze them, and check the vanishing of the associated integrals, one-by-one.

By Proposition \ref{orbit calc}, we have
\begin{equation}\label{unfold1}I^\Phi(\phi,s) = \int_{\GL_1(\Q) U_0(\Q)Z(\A) \backslash \GSp_6(\A)}{\phi(g)f^\Phi(\gamma_0 g,s)\,dg}\end{equation}
where the $\GL_1(\Q)$ is the set of elements of $\GSp_6(\Q)$ of the form $\mm{\lambda}{}{}{1}$, with $\lambda \in \Q^\times$, and $\gamma_0 \in \GG(\Q)$ satisfies $(0,0,0,1)\gamma_0 = f_\mathcal{O}$.  Indeed, this is the term from the open orbit, and the negligible orbits contribute zero by the cuspidality of $\phi$.

Now we integrate $\phi$ over $U_0(\Q)\backslash U_0(\A)$ and Fourier expand:
\[\int_{U_0(\Q)\backslash U_0(\A)}{\phi(ug)\,du} = \sum_{\gamma \in \GL_1(\Q)}{\phi_\chi(\gamma g)}.\]
Placing this Fourier expansion into (\ref{unfold1}) we get
\begin{equation}\label{unfold2}I^\Phi(\phi,s) = \int_{U_0(\A)Z(\A)\backslash \GSp_6(\A)}{\phi_\chi(g)f^\Phi(\gamma_0 g,s)\,dg}.\end{equation}
Integrating
\[\int_{U_0(\A)\backslash \GSp_6(\A)}{|\nu(g)|^s \phi_\chi(g)\Phi(f_\mathcal{O} g)\,dg}\]
over $Z(\A)$ one obtains (\ref{unfold2}), giving the theorem. 
\end{proof}

\begin{remark} In Theorem \ref{thm:unfold}, we could have chosen to unfold the global integral $I^\Phi(\phi,s)$ to any Fourier coefficient $\phi_{\chi'}(g)$, where $\chi'\left(\mm{1}{h}{}{1}\right) = \psi(\tr(T'h))$ and the quaternion algebra associated to $T'$ is $B$.  Using $T'$ instead of $T$ would have no effect on the results or proofs in section \ref{Unfolding}.  However, using $T$ that corresponds to the maximal order $B_0$ did make a difference in Corollary \ref{dirSeriesAIT}, and will also make a difference in Proposition \ref{charInt} below.\end{remark}

\section{Evaluation of the integral}\label{unram}
In this section we relate the unfolded Rankin-Selberg integral to the Spin $L$-function.
\subsection{The finite places}
Our first task is to compute the unipotent integral $\int_{U_0\backslash U_P}{\chi(u)\Phi(f_\O u g)\,du}.$  For an element $m \in \GL_3$, and an element $\lambda \in \GL_1$, we set 
\[(\lambda, m) = \left(\begin{array}{c|c}\lambda \det(m)\,^{t}m^{-1} & \\ \hline & m \end{array}\right)\]
an element of $\GSp_6$.  For $(\lambda, m) \in \GL_1(\A_f) \times \GL_3(\A_f)$, we define
\[\Xi(\lambda,m) := \begin{cases} 1 &\mbox{if } \lambda \in \widehat{\Z}, m \in M_3(\widehat{\Z}), \text{ and } m^{-1}T c(m) \in M_3(\widehat{\Z}) \\ 0 &\text{otherwise} \end{cases}.\]

\begin{proposition}\label{charInt} The integral
\begin{equation}\label{eq:charInt}\int_{U_0(\A_f) \backslash U_P(\A_f)}{\chi(u)\Phi_f(f_\O u(\lambda,m))\,du} = |\lambda|_f \Xi(\lambda,m).\end{equation} \end{proposition}
\begin{proof} Recall the lattice $J_0$ in $H_3(B)$ defined in (\ref{J_0}), and set $J_0(\widehat{\Z}) = J_0 \otimes_\Z \widehat{\Z}$.  Suppose $h = (\lambda,m)$.  One computes
\begin{equation}\label{f_O action}f_\mathcal{O}n(u)h = \left(0,0,m^{-1}A(T)c(m),\frac{\tr(Tu)}{\lambda}\right).\end{equation}
Hence for the integral (\ref{eq:charInt}) to be nonzero, one needs $m^{-1}A(T)c(m)$ to be in $J_0(\widehat{\Z})$ and the integral
\begin{equation}\label{abs lambda int}\int_{U_0(\A_f) \backslash U_P(\A_f)}{\psi(\tr(Tu)) \charf\left(\frac{\tr(Tu)}{\lambda} \in \widehat{\Z}\right)\,du}\end{equation}
to be nonzero. By Lemma \ref{AITlemma}, $m^{-1}A(T)c(m)$ is in $J_0(\widehat{\Z})$ if and only if $m \in M_3(\widehat{\Z})$ and $m^{-1}Tc(m) \in M_3(\widehat{\Z})$.  The integral (\ref{abs lambda int}) is zero when $\lambda \notin \widehat{\Z}$, and is $|\lambda|_f$ when $\lambda \in \widehat{\Z}$. The proposition follows.
\end{proof}

The global integral we must compute is
\begin{align*}I^\Phi(\phi,s) &= \int_{U_0(\R)\backslash \GSp_6(\R)}{|\nu(g_\infty)|_\infty^s \Phi_\infty(f_\O g_\infty)} \\ & \times \left(\int_{U_0(\A_f)\backslash \GSp_6(\A_f)}{\phi_\chi(g_f g_\infty) |\nu(g_f)|_f^s \Phi_f(f_\O g_f)\,dg_f}\right)\,dg_\infty \\& = \int_{U_0(\R)\backslash \GSp_6(\R)}{|\nu(g_\infty)|_\infty^s \Phi_\infty(f_\O g_\infty) I_f(\phi,s,g_\infty) \, dg_\infty}\end{align*}
where
\begin{align}\label{fteInt} I_f(\phi,s,g_\infty) &= \int_{U_0(\A_f)\backslash \GSp_6(\A_f)}{\phi_\chi(g_f g_\infty) |\nu(g_f)|_f^s \Phi_f(f_\O g_f)\,dg_f} \nonumber \\ &= \int_{\left(\GL_1 \times \GL_3\right)(\A_f)}{\delta_P^{-1}(h)|\det(m)|_f^{-1}\phi_\chi(\pi(h)g_\infty) |\nu(h)|_f^{s+1} \Xi(\lambda,m)\,dh}. \end{align}
Here $h = (\lambda, m)$, $\delta_P$ denotes the modulus character of the Siegel parabolic in $\GSp_6$, and we have applied Proposition \ref{charInt}, the Iwasawa decomposition, and the fact that $\lambda = \nu(h) \det(m)^{-1}$.

Translating between the classical and adelic language, the Dirichlet series of Corollary \ref{dirSeriesAIT} exactly evaluates the integral (\ref{fteInt}), as we check in the following proposition.
\begin{proposition}\label{unramCalc} Assume $\phi$ is as in Definition \ref{holAssumption}.  Then
\[I_f(\phi,s,g_\infty)=\frac{L(\pi,Spin,s-2)}{\zeta(2s-4)\zeta_{D_B}(2s-2)}\phi_\chi(g_\infty).\]
\end{proposition}
\begin{proof} If $\nu(g_\infty) > 0$, one has 
\begin{equation}\label{phichi} \phi_\chi(g_\infty) = a_f(T)\nu(g_\infty)^{-r}j(g_\infty,i)^{-2r}e^{2\pi i \tr(T g_\infty\cdot i)}, \end{equation}
while $\phi_\chi(g_\infty) = 0$ if $\nu(g_\infty) < 0$.  Suppose $h = (\lambda, m) \in M_{P,6}(\A_f)$.  Then $h = h_{\Q} h_\infty^{-1} k$ for some $h \in M_{P,6}(\Q)$ and $k \in M_{P,6}(\Z_p)$, where $h_\infty = \left(h_{\Q}\right)_\infty$ is the component of $h_\Q$ at the archimedean place.  Hence $\phi_\chi(h g_\infty) = \phi_\chi(h_\Q h_\infty^{-1} g_\infty) = \phi_{\chi \cdot h_\Q}(h_\infty^{-1} g_\infty)$, where $\chi \cdot h_\Q$ is the Fourier coefficient (\ref{phichi1}) except with $T$ replaced by $\lambda m^{-1} T c(m)$.  Thus,
\begin{align*} \phi_\chi(h g_\infty) &= a_f(\lambda m^{-1} T c(m)) \nu(h_\infty^{-1} g_\infty)^{-r} j(h_\infty^{-1} g_\infty, i)^{-2r} e^{2 \pi i \tr\left( (\lambda m^{-1} T c(m)) (h_\infty^{-1} g_\infty) \cdot i \right)} \\ &= a_f(\lambda m^{-1} T c(m)) \lambda^{-3r} \det(m)^{-r} \left( \nu(g_\infty)^{-r} j(g_\infty,i)^{-2r} e^{2 \pi i \tr(T g_\infty \cdot i)}\right). \end{align*}
Hence if $h \in (h_\Q)_f M_{P,6}(\Z_p)$, with $h_\Q = (\lambda, m)$, then the integrand in (\ref{fteInt})
\begin{align*} & =\left(\lambda^6 \det(m)^2\right)\det(m) \left(a_f(\lambda m^{-1} T c(m)) \lambda^{-3r} \det(m)^{-r} \right) (\lambda \det(m))^{-s-1} \Xi(\lambda,m) \\ & \qquad \times \left( \nu(g_\infty)^{-r} j(g_\infty,i)^{-2r} e^{2 \pi i \tr(T g_\infty \cdot i)}\right) \\ & = \lambda^{-(s+3r-5)} \det(m)^{-(s+r-2)}\Xi(\lambda,m)a_f(\lambda m^{-1}Tc(m)) \times \left( \nu(g_\infty)^{-r} j(g_\infty,i)^{-2r} e^{2 \pi i \tr(T g_\infty \cdot i)}\right) .\end{align*}
Applying Corollary \ref{dirSeriesAIT} and (\ref{phichi}) again gives the proposition.
\end{proof}

\subsection{The infinite place}\label{arch}
Suppose $\phi$ is a cusp form corresponding to a level one, holomorphic Siegel modular form of weight $2r$, as in Definition \ref{holAssumption}.  We now compute the normalized archimedean integral
\begin{equation*}I^*_{\infty,2r}(\phi,s) = \pi^{-(s+r)}D_B^s \Gamma_\R(2s+2r-4)\Gamma_\R(2s+2r-8) I^{\Phi}_{\infty,2r}(\phi,s),\end{equation*}
where
\begin{equation}\label{IPhi}I_{\infty,2r}^\Phi(\phi,s) = \int_{U_0(\R)\backslash \GSp_6(\R)}{\phi_\chi(g) |\nu(g)|^s \Phi_{\infty,2r}(f_{\mathcal{O}} g)\,dg}.\end{equation}
We prove
\begin{theorem}\label{archCalc} Suppose $\phi$ is as in Definition \ref{holAssumption}, and that $\phi$ corresponds to the level one Siegel modular form $f$.  Then $I^*_{\infty,2r}(\phi,s)$ equals 
\[a_f(T)\Gamma_\C(s+r-4)\Gamma_\C(s+r-3)\Gamma_\C(s+r-2)\Gamma_\C(s+3r-5)\]
up to a nonzero constant. \end{theorem}
Our first task is to define the Schwartz function $\Phi_{\infty,2r}$ on $W(\R)$.  To do this, we digress slightly and discuss the Hermitian symmetric space $\mathcal{H}$ associated to $\GG(\R)$.
\subsubsection{The Hermitian symmetric space for $\GG(\R)$}
We assume $B_\R =B \otimes \R$ is the Hamilton quaternions.  Set $B_\C = B_\R \otimes_\R \C$, with the involution on $B_\C$ the one from $B_\R$, extended $\C$-linearly to $B_\C$.  Associated to $B_\C$, we have the space $H_3(B_\C)$ of three-by-three Hermitian matrices over $B_\C$.  Elements of $H_3(B_\C)$ are formal expressions $Z = X + iY$, with $X, Y$ in $H_3(B_\R)$.  Note that the $i$ in this expression is not the $i$ in $B_\R$; this $i$ is in the center of $B_\C$.  The Hermitian symmetric space $\mathcal{H}$ is the set of $Z = X + iY$ in $H_3(B_\C)$ with $Y$ positive definite.

Denote by $\GG^+(\R)$ the subgroup of $\GG(\R)$ that is the connected component of the identity.  We explain the action of $\GG^+(\R)$ on $\mathcal{H}$.  First, note that we have an inclusion $\GG(\R) \rightarrow \GG(\C)$, and $\GG(\C)$ acts on $W(\C) = W(\R) \otimes_\R \C$.  Set $e = (1,0,0,0)$, a rank one element of $W$. Identify $Z \in \mathcal{H}$ with the rank one element 
\[r(Z) = e \, n(-Z) = (1,-Z,Z^\sharp,-N(Z)).\]
Consider the projective space $\mathbf{P}(W(\C)^{rk=1})$ of rank one $\C$-lines in $W(\C)$.  Then $\GG(\R)$ acts on this space on the left via $g \cdot v = vg^{-1}$.  We will see momentarily that $\GG^+(\R)$ preserves the image of $\mathcal H$ inside $\mathbf{P}(W(\C)^{rk=1})$, and acts on it transitively.  In fact, the factor of automorphy 
\[j_{\GG(\R)}(g,Z): \GG^+(\R) \times \mathcal{H} \rightarrow \C^\times\]
and the action of $\GG^+(\R)$ on $\mathcal{H}$ are simultaneously defined by the equality
\[r(Z)g^{-1} = j_{\GG(\R)}(g,Z)r(gZ).\]
We summarize what we need in a proposition.  Recall the rank one element $f = (0,0,0,1)$.
\begin{proposition}\label{Herm action} Suppose $Z \in \mathcal H$, and $g \in \GG^+(\R)$.  Then the complex number $j_{\GG(\R)}(g,Z) = \langle r(Z)g^{-1}, f \rangle$ is nonzero, and thus
\[r(Z)g^{-1} = j_{\GG(\R)}(g,Z)r(gZ)\]
for some $gZ$ in $H_3(B_\C)$.  The element $gZ$ is in $\mathcal H$, and thus this equality defines an action of $\GG^+(\R)$ on $\mathcal{H}$.  Under the natural embedding $\mathcal{H}_3 \rightarrow \mathcal{H}$, the two actions of $\GSp_6^+(\R)$ are the same.  Furthermore, $j_{\GSp_6}(g,Z) = j_{\GG(\R)}(g,Z)$ when $g \in \GSp_6^+(\R)$ and $Z \in \mathcal{H}_3$. \end{proposition}
\begin{proof} We sketch the proof.  First, a minimal parabolic subgroup $P_0$ of $\GG$ is discussed in section \ref{minPar}.  One can check that the subgroup $P_0^+(\R)$ of $\GG^+(\R)$ has $j_{\GG(\R)}(P_0^+(\R),\mathcal{H}) \neq 0$ and acts transitively on $\mathcal{H}$.  Now suppose $g \in \GG^+(\R)$, $Z \in \mathcal{H}$, and $p \in P_0^+(\R)$ satisfies $p \cdot i = Z$.  Then
\[\langle r(Z) g^{-1}, f \rangle = j_{\GG(\R)}(p,i)^{-1} \langle r(i) p^{-1} g^{-1}, f\rangle = \nu(gp)^{-1}j_{\GG(\R)}(p,i)^{-1} \langle r(i), fgp\rangle. \]
This last quantity is then nonzero by Lemma \ref{rk1 r(i)} below.  Hence $r(Z)g^{-1} = j_{\GG(\R)}(g,Z)r(gZ)$ for some $gZ$ in $H_3(B_\C)$.  Next, one verifies the equality $N(\iIm(W)) = \frac{1}{8i} \langle \overline{r(W)}, r(W) \rangle$ for any $W \in H_3(B_\C)$.  One deduces from this equality that $N(\iIm(gZ)) > 0$.  Then that $gZ \in \mathcal{H}$ follows from a continuity argument.

To check that the actions and factors of automorphy agree for $g \in \GSp_6^+(\R)$ and $Z \in \mathcal{H}_3$, it suffices to check this on generators $n(X), J_6$ and $M_{P,6}^+(\R)$, which may be done easily. \end{proof}

Before defining $\Phi_{\infty,2r}$, we need an additional lemma.  Recall the element $J$ of $\GG(\Q)$ that sends $(a,b,c,d) \mapsto (-d,c,-b,a)$.
\begin{lemma}\label{rk1 r(i)} If $v = (a,b,c,d)$, define $||v||^2 = \langle v, vJ \rangle$.  Then $||v||^2 = a^2 + \tr(b,b) + \tr(c,c) + d^2$, and hence if $v \in W(\R)$, $||v||^2 \geq 0$ and is only $0$ when $v = 0$.  If furthermore $v$ is of rank one, then $|\langle r(i),v \rangle|^2 = ||v||^2$. \end{lemma}
\begin{proof} The formula for $||v||^2$ is trivial.  Suppose $v$ is rank one.  We have $\langle r(i),v \rangle = (d-\tr(b)) + i(\tr(c)-a)$.  Hence 
\begin{align*} |\langle r(i),v \rangle|^2 &= (d-\tr(b))^2 + (a-\tr(c))^2 \\ &=d^2 -2d\tr(b) + \tr(b)^2 + a^2 - 2a\tr(c)+\tr(c)^2 \\ &= a^2 + (\tr(b)^2-2\tr(b^\sharp)) + (\tr(c)^2 - 2\tr(c^\sharp)) + d^2, \end{align*}
where in the last equality, since $v$ is rank one, $db = c^\sharp$ and $ac = b^\sharp$.  For general $W$ one has the equality $\tr(W)^2 - 2\tr(W^\sharp) = \tr(W,W)$.  The lemma follows. \end{proof} 

\subsubsection{Archimedean calculation}
Define 
\begin{equation}\label{Phi2r}\Phi_{\infty,2r}(v) = e^{-||v||^2}\langle r(i), v \rangle^{2r}.\end{equation}
This is obviously a Schwartz function on $W(\R)$.  Note that we have $\langle r(i), vgk\rangle = j(k,i) \langle r(i),vg\rangle $ for $k$ in $K_{\infty,\Sp_6}$.  Our assumption on $\phi$ gives
\[\phi_\chi(g) = a_f(T)\nu(g)^{-r}j(g,i)^{-2r}e^{2\pi i \tr(T g\cdot i)}\]
if $g \in \GSp_6^+(\R)$, and $\phi_\chi(g) = 0$ if $\nu(g) < 0$.  (The modular form associated to $\phi$ has \emph{negative definite} Fourier coefficients when restricted to the lower half-space, $\mathcal{H}_3^-$.)  Furthermore, for $g \in \GSp_6^+(\R)$, one has the identity
\begin{equation}\label{f_O eqn}\nu(g)^{-1}j(g,i)^{-1} \langle r(i), f_\mathcal{O} g \rangle = \tr\left(T (g\cdot i)\right).\end{equation}
Note that the left hand side of (\ref{f_O eqn}), as a function of $g$, is right $Z(\R)K_{\infty,\Sp_6}$-invariant.  Hence to prove the equality, it suffices to take $g$ in the Siegel parabolic of $\GSp_6$, and then (\ref{f_O eqn}) follows from (\ref{f_O action}).

Now, putting (\ref{Phi2r}) into (\ref{IPhi}) and integrating over the center $Z(\R)$, we obtain
\begin{align*}I^\Phi_{\infty,2r}(\phi,s) &= \Gamma(s+r) \int_{U^0(\R)Z(\R)\backslash \GSp_6(\R)}{\frac{|\nu(g)|^s \langle r(i), f_\mathcal{O} g \rangle^{2r} \phi_\chi(g)}{|\langle r(i), f_\mathcal{O} g \rangle|^{2s+2r}}\,dg}\\ & = \Gamma(s+r)\int_{U^0(\R)Z(\R)\backslash \GSp_6^+(\R)}{a(T)\frac{|\nu(g)|^{s+r} \tr(TZ)^{2r} e^{2 \pi i \tr(TZ)}}{|\langle r(i), f_\mathcal{O} g\rangle|^{2s+2r}}\,dg}, \end{align*}
where $Z = g \cdot i = X+iY$.  One has
\[\frac{|\nu(g)|^{s+r}}{|\langle r(i),f_\mathcal{O} g \rangle|^{2s+2r}} = \frac{|\nu(g)^{-1}j(g,i)^{-2}|^{s+r}}{|\tr(TZ)|^{2s+2r}} = \frac{N(Y)^{s+r}}{|\tr(TZ)|^{2s+2r}}.\]
Hence, $I^\Phi_{\infty,2r}(\phi,s) = a(T)\Gamma(s+r)N(T)^{-(s+r)}J(s)$, with
\begin{align}\label{J(s)} J(s) &= \int_{U^0(\R)Z(\R)\backslash \GSp_6^+(\R)}{\frac{N(TY)^{s+r}\tr(TZ)^{2r} e^{2 \pi i \tr(TZ)}}{|\tr(TZ)|^{2s+2r}}\,dg}\nonumber \\ &=  \int_{U^0_1(\R)Z(\R)\backslash \GSp_6^+(\R)}{\frac{N(Y)^{s+r}\tr(Z)^{2r} e^{2 \pi i \tr(Z)}}{|\tr(Z)|^{2s+2r}}\,dg},\end{align}
where $U^0_1 =\{\mm{1}{X}{}{1} : \tr(X) = 0\}.$  The change of variables implicit in equality (\ref{J(s)}) is $g = m_T^{-1} g'$, where $m_T = \mm{T^{1/2}}{}{}{T^{-1/2}}$ and $T^{1/2}$ a positive-definite symmetric squareroot of $T$.

We now explain how to compute the integral $J(s).$  As the global integral representation is an analogue of Garrett's triple product \cite{garrett}, we follow the archimedean calculation of \cite{garrett}.  We compute up to nonzero constants.  First, we have that 
\[J(s) = \int_{Z(\R)\backslash M_{P,6}^+(\R)}{N(Y)^{s+r-2}e^{-2\pi \tr(Y)}\int_\R{\frac{(x+i \tr(Y))^{2r}}{|x+i \tr(Y)|^{2s+2r}}e^{2\pi i x}\,dx}\,d^*Y}.\]
Here $d^*Y$ is the $M_{P,6}^+(\R)$ invariant measure on $Y$, and we have used that $\delta_P(g) = \det(Y)^2$ for $g \in M_{P,6}^+(\R)$.  Making a variable change, the inner integral becomes 
\[\tr(Y)^{1-2s} \int_\R{\frac{(x+i)^{2r}}{|x+i|^{2s+2r}}e^{2\pi i \tr(Y)x}\,dx},\]
using that $\tr(Y)$ is positive.  We now introduce a $\Gamma$-integral, and change the order of integration to obtain
\begin{align*} (2\pi)^{1-2s}\Gamma(2s-1)J(s) =&\, \int_{\R}{\frac{(x+i)^{2r}}{|x+i|^{2s+2r}}\int_{Y}{N(Y)^{s+r-2}e^{2 \pi i (x+i) \tr(Y)} }} \\ & \times \left(\int_{t \geq 0}{e^{-2\pi t \tr(Y)}t^{2s-1}\,\frac{dt}{t}}\right)\,d^*Y\,dx.\end{align*}
Define a logarithm on $\C$ by removing the nonpositive reals, and use this logarithm to define $z^s$ for such a $z$.  Then switching the order of integration we obtain
\begin{align*} (2\pi)^{1-2s}\Gamma(2s-1)J(s) =&\, \int_{\R}{\int_{t \geq 0}{\frac{(x+i)^{2r}}{|x+i|^{2s+2r}}(1+t-ix)^{-3(s+r-2)}t^{2s-1}}}\\
& {{\times \left(\int_{Y}{N(Y')^{s+r-2}e^{-2\pi\tr(Y')}d^*(Y')}\right)\,\frac{dt}{t}}\,dx},\end{align*}
where $Y' = (1+t-ix)Y$.
The integral over $Y$ is the so-called Siegel integral, and it is a nonzero constant times
\[2^{-(3s+3r-6)}\Gamma_\R(2s+2r-4)\Gamma_\R(2s+2r-5)\Gamma_\R(2s+2r-6).\]
See, for instance, \cite[Theorem VII.1.1]{fK} for this fact.

We are thus left to evaluate the integral
\[K(s,r) = \int_{\R}{\int_{t \geq 0}{\frac{(x+i)^{2r}}{|x+i|^{2s+2r}}(1+t-ix)^{-3(s+r-2)}t^{2s-1}\frac{dt}{t}}\,dx},\]
which may be done exactly as in \cite[Proof of Proposition 5.1]{garrett}.  Up to a nonzero constant, one obtains
\[K(s,r) = 2^{-(3s+3r-6)}\frac{\Gamma(s+3r-5)\Gamma(2s-1)}{\Gamma(s+r)\Gamma(2s+2r-5)}.\]
Thus, again up to nonzero constants,
\[J(s) = (2\pi)^{2s-1} 4^{-(3s+3r-6)}\frac{\Gamma(s+3r-5) }{\Gamma(s+r)\Gamma(2s+2r-5)} \Gamma_\R(2s+2r-4)\Gamma_\R(2s+2r-5)\Gamma_\R(2s+2r-6).\]
Via the duplication formula $\Gamma_\R(2s+2r-5)\Gamma_\R(2s+2r-4) = \Gamma_\C(2s+2r-5)$, $J(s)$ simplifies to
\[4^{-(3s+3r-6)}\frac{\Gamma(s+3r-5)}{\Gamma(s+r)}\Gamma_\R(2s+2r-6),\]
again up to nonzero constants. Hence $I^\Phi_{\infty,2r}(\phi,s) = a(T)D_B^{-s} 4^{-(2s+2r-6)}\Gamma(s+3r-5) \Gamma_\R(2s+2r-6)$, since $D_B = 4N(T)$ by Corollary \ref{4detT}.  Thus, up to a nonzero constant,
\[I_{\infty,2r}^*(\phi,s) = a(T) \Gamma_\C(s+r-4)\Gamma_\C(s+r-3)\Gamma_\C(s+r-2)\Gamma_\C(s+3r-5),\]
completing the proof of Theorem \ref{archCalc}.

\section{The group $\GG(\A)$ and its Siegel Eisenstein series}\label{Eisenstein}
The purpose of this section is to prove Theorem \ref{NormEis}, where $\Phi_{2r,\infty}$ is defined in (\ref{Phi2r}).  We will prove this theorem by first proving it for $E_{0}^*(g,s)$, the Eisenstein series that is spherical at all places.  Then we will define essentially a Maass-Shimura differential operator $\mathcal D$, and show that $\mathcal{D} E_{2r}^*(g,s) = E_{2r+2}^*(g,s)$.  With this identity, Theorem \ref{NormEis} for $r \geq 0$ implies the same statement for $r+1$.  Here, $\mathcal{D}$ is an element of the universal enveloping algebra of the Lie algebra of $\GG(\R)$.  To analyze $E_{0}^*(g,s)$, we apply Langlands' theorem of the constant term and his functional equation \cite{langlandsFE}.  In order to do the requisite intertwining operator calculations needed to apply these theorems of Langlands, we need some basic results about $\GG(\A)$, such as a description of the minimal parabolic subgroup and Iwasawa decomposition.  We begin with a summary of these results on $\GG(\A)$.  

\subsection{Minimal parabolic and Iwasawa decomposition}\label{minPar}
Suppose $F$ is a field of characteristic zero, $B$ is a quaternion algebra over $F$, $J = H_3(B)$, and $\GG$ is the reductive $F$-group that preserves the symplectic and quartic form on $W = F \oplus J \oplus J \oplus F$, up to similitude.  We describe a parabolic subgroup $P_0$ of $\GG$, which is minimal when $B$ is a division algebra, and the corresponding Iwasawa decomposition. 

For an element $x$ of $W$, denote by $a(x), b(x), c(x),d(x)$ the elements of $F, J,J, F$ respectively so that $x = (a(x),b(x),c(x),d(x)).$ If $b$ is in $H_3(B) \subseteq M_3(B)$, we denote $b_{ij}$ the element of $B$ in the $i,j$ position of $b$.  We denote by $e_{ij}$ the element of $M_3(B)$ that is $1$ at the $i,j$ position and zero elsewhere.

Denote by $S_4$ the subspace of $W$ consisting of $(a,b,c,d)$ with $a=0$, $b_{ij} = 0$ unless $i, j = 1$, $c_{1k} = c_{j1} = 0$ for all $j, k$.  Then $S_4$ is an eight-dimensional isotropic subspace of $W$.  Furthermore, every element $s \in S_4$ has rank at most $2$.  Inside of $S_4$ we consider the two dimensional subspace $S_2$ spanned by $(0,0,e_{33},0)$ and $(0,0,0,1)$.  Then every element of $S_2$ has rank at most $1$.  Finally, inside of $S_2$ we consider the line $S_1 = F(0,0,0,1)$.  (The reasoning for the numbering will become apparent shortly.) Set $\mathcal{F}_0$ to be the flag $S_4 \supseteq S_2 \supseteq S_1$, and define $P_0$ to be the subgroup of $\GG$ stabilizing $\mathcal F_0$.  Now, define $S_3$ to be the seven-dimensional space in $S_4$ consisting of the $x = (0,b,c,d)$ with $b=0$, i.e., $S_3 = S_4 \cap (0,0,*,*)$.  It can be shown that since $P_0$ stabilizes $(0,0,0,*)$, it automatically stabilizes $(0,0,*,*)$, and thus $P_0$ stabilizes $S_3$.

\begin{proposition}\label{typeC3} The subgroup $P_0$ is a parabolic subgroup of $\GG$.  Denote by $T$ the diagonal maximal torus in $\GSp_6$, considered inside of $\GG$.  If $B$ is not split over $F$, then $T$ is a maximal split torus of $\GG$, $P_0$ is a minimal parabolic subgroup, and $\GG$ has rational root type $C_3$.\end{proposition}

Now suppose $F = \Q$, and $B_\R$ is a division algebra.  We describe the associated Iwasawa decomposition of $\GG(\A)$.  Define $K_p$ to be the subgroup of $\GG(\Q_p)$ that acts invertibly on the lattice $W(\Z_p)$.  Define $K_\infty$ to be the subgroup of $\GG^{1,+}(\R)$ that commutes with $J$.  Here $\GG^1$ denotes the elements of $\GG$ with similitude equal to $1$.  Because $(v,w) := \langle v, wJ\rangle$ is an inner product on $W(\R)$, $K_\infty$ is compact.
\begin{proposition}\label{Iwasawa} The subgroup $K_\infty$ is the stabilizer of $i$ for the $\GG^{1,+}(\R)$ action on the Hermitian symmetric space $\mathcal{H}$.  For all places $v$ of $\Q$, one has the decomposition $\GG(\Q_v) = P_0(\Q_v)K_v$. \end{proposition}

We omit the proofs of Propositions \ref{typeC3} and \ref{Iwasawa}, as they appear to be well-known.

\subsection{The spherical Eisenstein series}
We are now ready to analyze the \emph{spherical} normalized Eisenstein series.  Recall we denote by $P$ the parabolic subgroup of $\GG$ that is the stabilizer of the line spanned by $f = (0,0,0,1)$.  If $p \in P$, then $p$ acts on this line as multiplication by some $d$ in $\GL_1$, and acts on the one-dimensional space $W/(f)^\perp$ as multiplication by some $a$ in $\GL_1$.  Define a character $\chi_s: P(\A)\rightarrow \C^\times$ as $\chi_s(p) = |a/d|^s$.  For every place $v$ of $\Q$, set $f_v(g,s)$ to be the unique element in $Ind_{P(\Q_v)}^{\GG(\Q_v)}(\chi_s)$ that is $1$ on $K_v$. The unnormalized spherical Eisenstein series is 
\[E_0(g,s) =\sum_{\gamma \in P(\Q) \backslash \GG(\Q)}{f(\gamma g,s)},\]
where $f(g,s) = \prod_{\text{ all } v}{f_v(g_v,s)}$ and the sum converges for $Re(s) >> 0$.  The normalized Eisenstein series $E_0^*(g,s)$ is
\begin{align}\label{normE0} E_0^*(g,s) =& \, D_B^{s}\left(\prod_{v < \infty \text{ split}}{\zeta_v(2s)\zeta_v(2s-2)\zeta_v(2s-4)}\right)\left(\prod_{v < \infty \text{ ramified}}{\zeta_v(2s)\zeta_v(2s-4)} \right) \\ & \times  \Gamma_\R(2s) \Gamma_\R(2s-4) \Gamma_\R(2s-8) E_0(g,s).\end{align}
Here $D_B = \prod_{p < \infty \text{ ramified}}{p}$ is the product of the finite primes $p$ for which $B \otimes \Q_p$ is not split.  By the theory of Langlands, $E_0^*$ has meromorphic continuation in $s$ and is a function of moderate growth where it is finite.  The purpose of this section is to describe the proof of the following theorem.
\begin{theorem}\label{SpherEis} The Eisenstein series $E_0^*(g,s)$ has at most finitely many poles, all contained in the set of integer and half-integer points in the interval $[0,5]$, and satisfies the functional equation $E_0^*(g,s) = E_0^*(g,5-s)$. \end{theorem} 
The proof will be an application of Langlands' theory of the constant term, and the Langlands' functional equation.  We begin with the computation of intertwining operators on $\GL_2(B)$.

For a finite prime $p$, set $K^B_p \subseteq \GL_2(B_p)$ to be the subgroup acting invertibly on the lattice $B_0(\Z_p) \oplus B_0(\Z_p)$.  Define $K^B_\infty$ to be the subgroup of $\GL_2(B_\infty)$ consisting of those $m \in \GL_2(B_\infty)$ with $m m^* = 1$.  Consider the function $f_v^B(g,s_1,s_2): \GL_2(B_v) \rightarrow \C$ defined by 
\[f^B_v\left(\mm{b_1}{*}{}{b_2}, s_1,s_2\right) = |n(b_1)|^{s_1+1}|n(b_2)|^{s_2-1}\]
and $f_v^B(K_v,s_1,s_2) = 1$.

\begin{definition} If $v < \infty$, and $B_v$ is split, define $\zeta_{B_v}(s) = \zeta_v(s)\zeta_v(s-1)$.  If $v < \infty$ and $B_v$ is ramified, define $\zeta_{B_v}(s) = \zeta_v(s)$.  If $B_\infty$ is ramified (as we assume throughout this section), define $\zeta_{B_\infty}(s) = D_B^{s/2}\Gamma_\C(s)$. \end{definition}

\begin{lemma}\label{GL2(B)} For $x$ in $B$, define $\overline{n}(x) \in \GL_2(B)$ to be the matrix $\mm{1}{}{x}{1}$.  Suppose the measure on $B(\A)$ is normalized so that $B(\Q) \backslash B(\A)$ has measure one, and $B_0(\Z_p)$ has measure one for all finite primes, both split and ramified.  Then, if $Re(s_1 -s_2) >>0$,
\[\int_{B_v}{f_v^B(\overline{n}(x),s_1,s_2)\,dx} = \frac{\zeta_{B_v}(s_1-s_2)}{\zeta_{B_v}(s_1-s_2 + 2)}.\] \end{lemma}
\begin{proof} Write $e, f$ for the basis of $B^2$.  For a finite prime $v$, define $\Phi_v$ on $B_v^2$ to be the characteristic function of $B_0(\Z_v)^2$.  Define $\Phi_\infty(b_1 e + b_2 f) = e^{-2\pi (n(b_1) + n(b_2))}$.  Note that, for all $v$, $\Phi_v(b_1 e+ b_2 f) = \Phi_v(b_1 e) \Phi_v(b_2 f)$, and that $\Phi_v(w K_v) = \Phi_v(w)$ for all $w$ in $B_v^2$.  Set $s = s_1 - s_2$.  Consider
\[f^\Phi(g,s_1,s_2) = |\det(g g^*)|^{s_1+1}\int_{B_v^\times}{\Phi_v(yfg) |n(y)|^{s+2}\,dy}.\]
Then, the transformation properties of $f^\Phi(g,s_1,s_2)$ implies the identity 
\[f^\Phi(g,s_1,s_2) = f^\Phi(1,s_1,s_2) f_v^B(g,s_1,s_2).\]
We compute
\begin{align*} \int_{B_v}{f^\Phi(\overline{n}(x),s_1,s_2)\,dx} =& \int_{B_v^\times}\int_{B_v}{\Phi(y(xe+f))|n(y)|^{s+2}\,dx}\,dy \\ =& \int_{B_v^\times}|n(y)|^{s+2} \Phi_v(yf) \int_{B_v}{\Phi_v(yxe)\,dx}\,dy \\ =& \left(\int_{B_v^\times}{|n(y)|^{s}\Phi_v(yf)\,dy}\right) \left(\int_{B_v}{\Phi_v(xe)\,dx} \right)\end{align*}
Hence
\[\int_{B_v}{f_v^B(\overline{n}(x),s_1,s_2)\,dx} = \frac{\left(\int_{B_v^\times}{|n(y)|^{s}\Phi_v(yf)\,dy}\right) \left(\int_{B_v}{\Phi_v(xe)\,dx} \right)}{\left(\int_{B_v^\times}{|n(y)|^{s+2}\Phi_v(yf)\,dy}\right)}.\]
When $v < \infty$, the first integral in the numerator of this fraction gives $\zeta_{B_v(s)}$ and the second integral gives $1$. When $v = \infty$, the first integral in the numerator gives $\Gamma_\C(s)$, up to a nonzero constant.  Via Lemma \ref{measInfinity}, the second integral gives $D_B^{-1}$.  The lemma follows.
 \end{proof}

Since $B_\R$ is the Hamilton quaternions, we have $B_\R = \R \oplus \R i \oplus \R j \oplus \R k$ with $i^2 = j^2 = k^2 = -1$, and $ij = k$.
\begin{lemma}\label{measInfinity} Suppose $\mu_{std}$ is the measure on $B_\R$ that gives the box spanned by $1, i , j, k$ measure one, and $\mu_{norm}$ is the measure determined by the statement of Lemma \ref{GL2(B)}.  Then $\mu_{norm} = 4D_B^{-1} \mu_{std}$.\end{lemma}
\begin{proof} The normalized measure is characterized by $\mu_{norm}\left(B_0 \backslash B_\R \right) = 1$.  Suppose $B_0$ corresponds to the half-integral symmetric matrix $T$.  Then the covolume of $B_0$ in $B_\R$ in $\mu_{std}$ is $\det(T)$.  Hence $\mu_{norm} = \det(T)^{-1} \mu_{std}$. Since $D_B = 4 \det(T)$ by Corollary \ref{4detT}, we are done. \end{proof}

Denote by $M_0 \subseteq P_0$ the centralizer of the maximal split torus $T$ defined in Proposition \ref{typeC3}.  We now describe some characters on $M_0(\A)$.  First, the weight spaces for the action of $T$ on $W$ are exactly the coordinate spaces, i.e., the $a(v), b_{ij}(v), c_{ij}(v), d(v)$, with $i \leq j$.  Since $T$ acts on these spaces with distinct weights, $M_0$ is exactly the subgroup of $\GG$ preserving the coordinate spaces.  

Consistent with the notation in (\ref{h def}), we write $a_i(b(v)) = b_{i+1,i-1}(v)$ and $c_i(b(v)) = b_{i,i}(v)$, with the indices in the set $\{1,2,3\}$ and taken modulo $3$.  If $m \in M_0$, one can verify easily that $m$ preserves the norm on the coordinate spaces $a_i(b(v))$, up to scaling.  For $m \in M_0$, define $m_i$ to be the endomorphism of $B = a_i(\{b(v): v \in W\})$ so that $a_i(b(vm)) = \nu(m) (a_i(b(v)))m_i$ for all $v$ in $W$. (In fact, one can show that $\nu$ and the $m_i$ determine $m$.)  Since $m_i$ preserves the norm on $B$ up to scaling, define $|m_i|$ by the equality $|n(b m_i)|=|m_i||n(b)|$, $b \in B$.  Similarly, define $x_i(m)$ by the equality $c_i(b(v m)) = \nu(m) x_i(m) c_i(b(v))$.  One immediately sees $|m_1| = |x_2 x_3|$, $|m_2| = |x_1 x_3|$, $|m_3| = |x_1x_2|$. 

Recall the elements $e = (1,0,0,0)$ and $f = (0,0,0,1)$ of $W$.  Since $M_0$ preserves the quartic form $Q$ up to similitude squared, $fm = (\nu(m)x_1x_2x_3)^{-1}f$.  Since $M_0$ preserves the symplectic form up to similitude, $e m = \nu(m)^2 x_1 x_2 x_3 e$.  Thus $\chi_s(m) = |\nu(m)|^{3s}|m_1|^s|m_2|^s|m_3|^s$.  One has the formula $\delta_{P_0}(m) =|\nu(m)|^{15} |m_1|^9|m_2|^5 |m_3|.$  Set 
\[\lambda_s(m) = |\nu(m)|^{3s-15/2}|m_1|^{s-9/2}|m_2|^{s-5/2}|m_3|^{s-1/2},\]
so that $\chi_s(m) = \lambda_s(m)\delta_{P_0}^{1/2}(m)$.

To understand the analytic properties of $E_0^*(g,s)$, it suffices to restrict to $\GG^1$, so from now on we drop the $\nu$'s.  From Proposition \ref{typeC3}, $\GG$ is of type $C_3$; let us write $\pm u_i \pm u_j,$ $i \neq j$, $\pm 2 u_i$ for its root system determined by the unipotent radical of $P_0$.  Before proceeding to the intertwining operator calculations, we state one more preliminary fact.  Consider the space of characters $\operatorname{Char}(M_0(\A))^{|\cdot |}$ of $M_0(\A)$ that factor through the maps $|m_i|$; i.e., this is the space of characters consisting exactly of the maps
\[m \mapsto |m_1|^{s_1}|m_2|^{s_2}|m_3|^{s_3}\]
for $s_1, s_2, s_3 \in \C$.  We write $s_1 u_1 + s_2 u_2 + s_3 u_3$ for this character.  Denote by $\WW = (\pm 1)^3 \rtimes S_3$ the Weyl group of $\Sp_6$, embedded in $\GG$.  We identify $\WW$ with the subset of $\Sp_6$ that preserves the set $\{\pm e_1, \pm e_2, \pm e_3, \pm f_1, \pm f_2, \pm f_3\}$.  Using the action of $\WW$ on $\wedge^3(W_6)$, one checks easily that $\WW$ preserves $\operatorname{Char}(M_0(\A))^{|\cdot |}$ and acts on it in the usual way; e.g., the transposition $(12)$ in $S_3$ takes $s_1 u_1 + s_2 u_2 + s_3 u_3$ to $s_2 u_1 + s_1 u_2 + s_3 u_3$ and the inversion $(-1,1,1)$ in $(\pm 1)^3$ takes this character to $-s_1 u_1 + s_2 u_2 + s_3 u_3$.  

Restricting characters in $\operatorname{Char}(M_0(\A))^{|\cdot |}$ to $T(\A)$ and dividing by two gives an identification
\[\operatorname{Char}(M_0(\A))^{|\cdot |} \rightarrow X_{un}^*(T(\A)) \simeq X^*(T) \otimes_\Z \C,\]
where $X_{un}^*(T(\A))$ denotes the characters on $T(\A)$ that factor through $| \cdot |$.  Thus, if $\alpha^\vee$ is a coroot for $\Sp_6$, and $\mu$ is a character of $M_0(\A)$ in $\operatorname{Char}(M_0(\A))^{|\cdot |}$, we may write $\langle \alpha^\vee, \mu \rangle$.  For example, for the character $\mu = s_1 u_1 + s_2 u_2 + s_3 u_3$, and $\alpha^\vee = (u_1 + u_2)^\vee$, $\langle \alpha^\vee, \mu \rangle = s_1 + s_2$.

Now we describe the constant term and functional equation computations.  For the functional equation, set $w_8 \in \WW$ to be the element represented by 
\[w_8 = \left(\begin{array}{ccc|ccc} & & & & & 1 \\ & & & &1& \\ & & & 1& & \\ \hline & &-1& & & \\ &-1 & & & & \\ -1& & & & & \end{array}\right) \in \Sp_6,\]
so that $w_8(e_i) = f_{4-i}$ and $w_8(f_i) = -e_{4-i}$.  Then for the spherical section $f(g,s)$ in $Ind_{P_0(\A)}^{\GG(\A)}(\chi_s)$, define
\[M(f)(g,s) = \int_{U_P(\A) \cong H_3(B)(\A)}{f(w_8^{-1}n(X)g,s)\,dX}\]
for $Re(s)$ sufficiently large.  The measure is normalized so that every closed connected subgroup $N$ of $U_0$ has $N(\Q) \backslash N(\A)$ given measure one.  Then $M(f)$ is a spherical section in $Ind_{P_0(\A)}^{\GG(\A)}(\chi_{5-s})$.  Indeed, $w_8$ is the Weyl group element taking $u_1 \mapsto -u_3$, $u_2 \mapsto -u_2$, $u_3 \mapsto -u_1$, and thus that $M(f)$ is in $Ind_{P_0(\A)}^{\GG(\A)}(\chi_{5-s})$ follows immediately from the computation of $\delta_{P_0}$ above.  That $M(f)$ is spherical of course follows from the Iwasawa decomposition.  Hence $M(f_s)$ is equal to $c(s)f(g,5-s)$ for some function $c(s)$.  Langlands' functional equation of the Eisenstein series is $E_0^*(g,f_s) = E_0^*(g, M(f_s))$.  We will determine $c(s)$ shortly, and then use this to check the functional equation stated in Theorem \ref{SpherEis}.

Now for the constant term of the Eisenstein series.  We must compute
\[E_0^{*,P_0}(g,s):=\int_{U_0(\Q)\backslash U_0(\A)}{E_0^*(ug,s)\,du}.\]
Langlands' theory of the constant term says that the Eisenstein series has analytic continuation to the same region as that of the constant term.  For a character $\chi$ in $\operatorname{Char}(M_0(\A))^{|\cdot |}$, denote by $\Phi^{\GG}_\chi$ the unique spherical element of $Ind_{P_0(\A)}^{\GG(\A)}(\delta_{P_0}^{1/2} \chi)$. For $w \in \WW$, set
\[A(w,\chi)(g) := \int_{N_w(\A)}{\Phi^{\GG}_\chi(w^{-1}ng)\,dn}\]
when it is absolutely convergent.  Here $N_w$ is the group $\prod_{ \alpha > 0, w^{-1}(\alpha) < 0}{U_{\alpha}}$, and $U_\alpha$ is the root group in $\GG$ corresponding to the root $\alpha$.  Then $A(w,\chi)(g)$ is in $Ind_{P_0(\A)}^{\GG(\A)}(\delta_{P_0}^{1/2} w(\chi))$, and is equal to $c_w(\chi)\Phi^{\GG}_{w(\chi)}$ for a certain factorizable function $c_w = \prod_v{c^v_w}$ on $\operatorname{Char}(M_0(\A))^{|\cdot |}.$  By the usual argument \cite[equation (5.6)]{gpsrDoubling}, one has
\[E_0^{P_0}(g,s) = \sum_{w \in \Omega}{c_w(\lambda_s)\Phi^{\GG}_{w(\lambda_s)}(g)},\]
where $\Omega$ is a set of representatives in $\WW$ for the double coset $P_0(\Q) \backslash \GG(\Q) \slash P(\Q).$

This double coset has eight elements.  The eight Weyl group elements $w_i \in \Omega$ can be obtained from \cite[Lemma 5.1]{gpsrDoubling}.  They are indexed by the sets $S_i$ of negative roots $\beta$ for which $w_i(\beta)$ is positive.  The sets are
\begin{enumerate}
\item $S_1 = \varnothing$; $c^v_1(s) = 1$.
\item $S_2 = \{-2u_3\}$; $c_2^v(s) = \frac{\zeta_v(2s-1)}{\zeta_v(2s)}$.
\item $S_3 = S_2 \cup \{-(u_2+u_3)\}$; $c_3^v(s) = c_2^v(s) \frac{\zeta_{B_v}(2s-3)}{\zeta_{B_v}(2s-1)}$.
\item $S_4 = S_3 \cup \{-2u_2\}$; $c_4^v(s) = c_3^v(s) \frac{\zeta_v(2s-5)}{\zeta_v(2s-4)}$.
\item $S_5 = S_3 \cup \{-(u_1+u_3)\}$; $c_5^v(s) = c_3^v(s) \frac{\zeta_{B_v}(2s-5)}{\zeta_{B_v}(2s-3)}$.
\item $S_6 = S_5 \cup \{-2u_2\}$; $c_6^v(s) = c_5^v(s) \frac{\zeta_v(2s-5)}{\zeta_v(2s-4)}$.
\item $S_7 = S_6 \cup \{-(u_1+u_2)\}$; $c_7^v(s) = c_6^v(s) \frac{\zeta_{B_v}(2s-7)}{\zeta_{B_v}(2s-5)}$.
\item $S_8 = S_7 \cup \{-2u_1\} = \overline{U_P}$; $c_v(s) = c_8^v(s) = c_7^v(s) \frac{\zeta_{v}(2s-9)}{\zeta_{v}(2s-8)}$. \end{enumerate}

We have also listed the functions $c_i^v(s) := c_{w_i}^v(\lambda_s)$.  Here, for $v = \infty$, $\zeta_v(s)$ means $\Gamma_\R(s)$.  The computation of the functions $c_i(s)$ proceeds by the usual Gindikin-Karpelevich argument of factorization of intertwining operators and reduction to rational rank one.  Recall that if the length of $w \in \WW$ is $r$, and $w = s_1 s_2 \cdots s_r$ is an expression of $w$ as a product of $r$ reflections, $s_i$ the reflection in the simple root $\alpha_i$, then
\begin{equation}\label{cProduct}c(w,\chi) = \prod_{k=1}^r \left(\int_{U_{\alpha_k}(\A)}{\Phi^{\GG}_{s_{k+1}\cdots s_r(\chi)}(s_k^{-1}n)\,dn}\right).\end{equation}

If $\alpha$ is a short simple root, then the root groups $U_\alpha$ and $U_{-\alpha}$ sit inside a $\GL_2(B)/\mu_2$ in $\GG$, that acts on $W$ via the maps $m(t)$ of section \ref{sec:Freud}.  Now, one checks that if $K_v^B$ is the compact subgroup of $\GL_2(B_v)$ defined above Lemma \ref{GL2(B)}, then $K_v^B/\mu_2$ sits inside $K_v$, for both of these short simple roots, and for all places $v$ of $\Q$.  It thus follows from Lemma \ref{GL2(B)} that if $\alpha_k$ is a short simple root,
\[ \int_{U_{\alpha_k}(\Q_v)}{\Phi^{\GG}_{s_{k+1}\cdots s_r(\lambda_s)}(s_k^{-1}n)\,dn} = \frac{ \zeta_{B_v}(\langle \alpha_k^\vee, s_{k+1}\cdots s_r(\lambda_s)\rangle)}{ \zeta_{B_v}(\langle \alpha_k^\vee, s_{k+1}\cdots s_r(\chi)\rangle +2)} = \frac{\zeta_{B_v}(\langle \beta_k^\vee, \lambda_s\rangle)}{\zeta_{B_v}(\langle \beta_k^\vee, \lambda_s\rangle +2)}\]
where $\beta_k = s_r \cdots s_{k+1}(\alpha_k)$.

If $\alpha_k$ is long root, the analogous formula is
\[ \int_{U_{\alpha_k}(\Q_v)}{\Phi^{\GG}_{s_{k+1}\cdots s_r(\lambda_s)}(s_k^{-1}n)\,dn} = \frac{\zeta_v(2 \langle \beta_k^\vee, \lambda_s \rangle)}{\zeta_v(2 \langle \beta_k^\vee, \lambda_s \rangle +1)}.\]
This formula is even easier to establish, because the long root subgroups sit inside $\Sp_6$, so the computation may done in an $\SL_2$ sitting inside $\Sp_6$.

For $w = s_1 \cdots s_r$, the roots $-\beta_k$ are exactly the negative roots $\beta$ for which $w(\beta) > 0$.  The computation of the $c_i^v(s)$ above now follows.  Putting in the normalization from (\ref{normE0}), one now checks without much effort that the constant term of the normalized Eisenstein series $E_0^*(g,s)$ has at worst finitely many poles, and they all occur at the integral and half-integral points in the interval $[0,5]$.  This completes the proof of the statement about the poles of $E_0^*(g,s)$.

Now for the functional equation.  Recall $M(f_s) = c(s)f_{5-s}$.  Then $c(s) = c_8(s)$ is determined above.  For $v < \infty$ split, one finds
\[c_v(s) = \frac{  \zeta_v(2s-5)\zeta_v(2s-7)\zeta_v(2s-9)}{\zeta_v(2s)\zeta_v(2s-2)\zeta_v(2s-4)}.\]
For $v < \infty$ ramified, one gets
\[c_v(s) = \frac{  \zeta_v(2s-5)\zeta_v(2s-7)\zeta_v(2s-9)}{\zeta_v(2s)\zeta_v(2s-4)\zeta_v(2s-8)}.\]
Using the duplication formula $\Gamma_\R(z) \Gamma_\R(z+1) = \Gamma_\C(z)$, one obtains
\[c_\infty(s) = D_B^{-3}\frac{\Gamma_\C(2s-5)\Gamma_\C(2s-7)\Gamma_\C(2s-9)}{\Gamma_\R(2s)^2 \Gamma_\R(2s-4)^2 \Gamma_\R(2s-8)^2}.\]
Now, we have 
\[E_0^*(g,s) = D_B^{s}\Lambda_\Q(2s)\Lambda_\Q(2s-2) \Lambda_\Q(2s-4)\left(\prod_{v < \infty \text{ ramified}}{\frac{1}{\zeta_v(2s-2)}} \right) \frac{\Gamma_\R(2s-8)}{\Gamma_\R(2s-2)} E_0(g,s)\]
where $\Lambda_{\Q}(s) = \prod_v{\zeta_v(s)}$ denotes the completed Riemann zeta function.  Thus, since $\Lambda_\Q(s) = \Lambda_\Q(1-s)$, $E_0^*(g,5-s)$
\[ = D_B^{5-s}\Lambda_\Q(2s-9)\Lambda_\Q(2s-7) \Lambda_\Q(2s-5)\left(\prod_{v < \infty \text{ ramified}}{\frac{1}{\zeta_v(8-2s)}} \right) \frac{\Gamma_\R(2-2s)}{\Gamma_\R(8-2s)} E_0(g,5-s).\]
The intertwining operator gives $E_0^*(g,s)$
\[ = D_B^{s-3}\Lambda_\Q(2s-9)\Lambda_\Q(2s-7) \Lambda_\Q(2s-5)\left(\prod_{v < \infty \text{ ramified}}{\frac{1}{\zeta_v(2s-8)}} \right) \frac{\Gamma_\R(2s-6)}{\Gamma_\R(2s)} E_0(g,5-s).\]
Thus
\[\frac{E_0^*(g,5-s)}{E_0^*(g,s)} = D_B^{8-2s} \left(\prod_{v < \infty \text{ ram}}{\frac{\zeta_v(2s-8)}{\zeta_v(8-2s)}}\right) \frac{\Gamma_\R(2-2s)\Gamma_\R(2s)}{\Gamma_\R(2s-6)\Gamma_\R(8-2s)}.\]
But $\zeta_v(s)/\zeta_v(-s) = -p^s$, and $\Gamma_\R(2z) \Gamma_\R(2-2z) = \frac{1}{\sin(\pi z)}$.
Thus the above fraction is
\[D_B^{8-2s}\left( \prod_{v < \infty \text{ ram}}{-p^{2s-8}}\right) \frac{\sin(\pi(s-3))}{\sin(\pi s)} = 1\]
since $B$ is ramified at an even number of places.  This completes the proof of Theorem \ref{SpherEis}.

\subsection{Differential Operators}
In this section we analyze the normalized Eisenstein series $E^*_{2r}(g,s)$ that is used in the Rankin-Selberg integral for a Siegel modular form on $\GSp_6$ of even weight $2r$.  In all that follows, $r \geq 0$ is an integer.  This Eisenstein series is defined as follows.  At finite places $p < \infty$, define $f_p^*(g,s)$ the normalized section as above, so that $f_p^*(K_p,s) = \zeta_p(2s)\zeta_p(2s-4)$ if $p$ is ramified for $B$, and $f_p^*(K_p,s) = \zeta_p(2s)\zeta_p(2s-2)\zeta_p(2s-4)$ if $p$ is split for $B$.  At infinity, the normalized section is defined as follows.  Define $f_{\infty,2r}$ to be the unique element of $Ind_{P(\R)}^{\GG(\R)}(\chi_s)$ that satisfies $f_{\infty,2r}(k,s) = j(k,i)^{2r}$ for all $k \in K_\infty$.  Then the normalized section 
\[f_{\infty,2r}^*(g,s) = \Gamma_\R(2s+2r)\Gamma_\R(2s+2r-4)\Gamma_\R(2s+2r-8)f_{\infty,2r}(g,s).\]
Finally, 
\[E_{2r}^*(g,s)= D_B^s \sum_{\gamma \in P(\Q) \backslash G(\Q)}{f_{2r}^*(\gamma g,s)},\]
where $f_{2r}^*(g,s) = f_{\infty,2r}^*(g,s)\prod_{v < \infty}{f_v^*(g,s)}$.  Thus when $r=0$, we recover the spherical normalized Eisenstein series.

We will use (essentially) the Maass-Shimura differential operators \cite[\S 19]{maass}, \cite{shimuraDiff} to obtain the desired analytic properties of $E_{2r}^*(g,s)$ from $E_{0}^*(g,s)$, as stated in the following theorem.
\begin{theorem}\label{DE2r} There is an element $\mathcal D$ of $\mathcal{U}(\mathfrak g)\otimes \C$, the complexified universal enveloping algebra of the Lie algebra of $\GG(\R)$, such that $\mathcal{D} E_{2r}^*(g,s) = E_{2r+2}^*(g,s)$.  Consequently, $E_{2r}^*(g,s) = E_{2r}^*(g,5-s)$ and $E_{2r}^*(g,s)$ has at worst prescribed poles at the integral and half-integral points in the interval $[0,5]$. \end{theorem}

The proof will work as follows.  First, recall that $j(g,i) = \langle r(i) g^{-1}, f \rangle$, so that $j( tg,i) = t^{-1}j(g,i)$ if $t \in \GL_1$ denotes, by slight abuse of notation, the element of $\GG$ that acts as multiplication by the scalar $t$ on $W$.  Define $J(g,i) = \nu(g) j(g,i) = \langle r(i), f g \rangle$, and for $v \in W(\R)$, recall that $\Phi_{\infty,2r}(v) = e^{-||v||^2}\langle r(i), v \rangle^{2r}$.  Then
\[\Phi_{\infty,2r}(fg) = e^{-|J(g,i)|^2}J(g,i)^{2r}.\]
One has
\begin{equation}\label{fintDef}f_{\infty,2r}^*(g,s) = \pi^{-(3s+3r-6)}\Gamma(s+r-2)\Gamma(s+r-4)|\nu(g)|^s \int_{\GL_1(\R)}{\Phi_{2r}(tfg)|t|^{2s}\,dt}.\end{equation}
Indeed, the function of $g$ on the right-hand side of (\ref{fintDef}) is in the correct induction space, has the correct transformation with respect to $K_\infty$, and at $g=1$ is $\pi^{-(3s+3r-6)}\Gamma(s+r)\Gamma(s+r-2)\Gamma(s+r-4)$ since the inner integral at $g=1$ is $\Gamma(s+r)$.  Hence this $f_{\infty,2r}^*$ agrees with normalized section used in the body of the paper.  In different notation,
\[f_{\infty,2r}^*(g,s) = \pi^{-(3s+3r-6)}\Gamma(s+r)\Gamma(s+r-2)\Gamma(s+r-4)|\nu(g)|^s \frac{J(g,i)^{2r}}{|J(g,i)|^{2s+2r}}.\]
To compute the action of $\mathcal D$ on $f_{\infty,2r}^*(g,s)$ and thus $E_{2r}^*(g,s)$, we compute it on the function of $g$ $g \mapsto \Phi_{\infty,2r}(f g)$.  We will find
\begin{equation}\label{DPhi2r}\mathcal{D}\Phi_{\infty,2r}(fg) = \pi^{-3}e^{- |J(g,i)|^2}\left\{|J(g,i)|^4 -9 |J(g,i)|^2+15\right\}J(g,i)^{2r+2}. \end{equation}
Now since
\[\int_{\GL_1(\R)}{e^{- t^2|J(g,i)|^2}|J(g,i)|^{2v}J(g,i)^{2r+2}|t|^{2s+2r+2+2v}\,dt} = \Gamma(s+r+1+v) \frac{J(g,i)^{2r+2}}{|J(g,i)|^{2s+2r+2}},\]
(\ref{DPhi2r}) implies
\begin{align*} \int_{\GL_1(\R)}{\mathcal{D}\Phi_{\infty,2r}(tg)|t|^{2s}\,dt} &= \pi^{-3}\left\{\Gamma(s+r+3)-9\Gamma(s+r+2)+15\Gamma(s+r+1)\right\}\frac{J(g,i)^{2r+2}}{|J(g,i)|^{2s+2r+2}} \\ &= \pi^{-3}\left\{ (s+r+2)(s+r+1)-9(s+r+1)+15 \right\}\Gamma(s+r+1) \\ & \qquad \qquad \times \frac{J(g,i)^{2r+2}}{|J(g,i)|^{2s+2r+2}} \\ &= \pi^{-3}(s+r-2)(s+r-4)\Gamma(s+r+1) \frac{J(g,i)^{2r+2}}{|J(g,i)|^{2s+2r+2}}. \end{align*}
Thus Theorem \ref{DE2r} follows as soon as we establish equation (\ref{DPhi2r}).

Let us now define $\mathcal{D}$ and prove the identity (\ref{DPhi2r}).  Set $h = n(i)\overline{n}(\frac{i}{2}) \in \GG(\C)$. Then $r(i) h = e$ and $r(-i) h = -8i f$.  Denote by $\overline{\mathfrak{u}_P}(\R)$ the Lie algebra of the unipotent subgroup of $\GG(\R)$ consisting of the $\overline{n}(X)$, for $X \in J_\R$, with $J = H_3(B)$.  We will define $\mathcal{D}_0$ in $\mathcal{U}(\overline{\mathfrak{u}_P}(\R)) \simeq Sym(J_\R)$, the universal enveloping algebra of $\overline{\mathfrak{u}_P}(\R)$, and then $\mathcal{D}$ will be given by $(2\pi i)^{-3} h \mathcal{D}_0 h^{-1}$.  Here we identify $J$ with $\overline{\mathfrak{u}_P}$ via the logarithm $\overline{L}$ of the map $\overline{n}$.  One has
\[(a,b,c,d) \overline{L}(X) = (\tr(b,X),c \times X, dX,0).\]

To define $\mathcal D_0$, we first rewrite the formula for the norm of an element of $J$ in coordinates.  For $a \in B_\R$, denote by $a(1), a(i), a(j), a(k)$ the real numbers defined by the equality $a = a(1) + a(i)i + a(j)j + a(k)k$.  If
\[x = \left(\begin{array}{ccc} c_1 & a_3 & a_2^* \\ a_3^* & c_2 & a_1 \\ a_2 & a_1^* & c_3 \end{array} \right),\]
then
\begin{align}\label{N(x)}N(x) =&\, c_1 c_2 c_3 - \left(\sum_{i=1,2,3}{\sum_{u \in \{1,i,j,k\}}{c_i a_i(u)^2}}\right) + 2a_1(1)a_2(1)a_3(1)\\ &\, - 2\left( \sum_{i=1,2,3}{\sum_{u \in \{i,j,k\}}{a_i(1)a_{i+1}(u)a_{i+2}(u)}}\right) - 2 \det(a_j(u)). \nonumber \end{align} 
Here 
\[(a_j(u)) = \left(\begin{array}{ccc} a_1(i) & a_2(i) & a_3(i) \\ a_1(j) & a_2(j) & a_3(j) \\ a_1(k) & a_2(k) & a_3(k) \end{array}\right).\]

View the coordinates $c_i, a_j(u)$ on the right hand side of (\ref{N(x)}) as elements of $J^\vee$, the dual of $J$, so that the expression becomes an element of $Sym^3(J^\vee)$.  The trace form $\tr( \;,\;)$ gives an identification $J \simeq J^\vee$, and thus the norm gives us an element $\mathcal D_0$ of $Sym^3(J) \subseteq \mathcal{U}(\overline{\mathfrak{u}_P}) \subseteq \mathcal{U}(\mathfrak g)$.  This is the differential operator $\mathcal D_0$ defined above.  More explicitly, recall that $e_{ij}$ denotes the element of $M_3(B)$ with a $1$ at position $i,j$ and zeroes elsewhere.  Define $v_1(u) = u e_{23} + u^*e_{32}$, and similarly for $v_2(u)$, $v_3(u)$.  Then under the identification of $J^\vee$ with $J$, $c_i$ becomes $e_{ii}$ and $a_i(u)$ becomes $v_i(u)/2$. Hence
\begin{align}\label{D_0} \mathcal D_0 =&\, e_{11}e_{22}e_{33} - \frac{1}{4}\left( \sum_{1 \leq i \leq 3, u \in \{1,i,j,k\}}{e_{ii}v_i(u)^2}\right) + \frac{1}{4}v_1(1)v_2(1)v_3(1)\\ &\, - \frac{1}{4}\left( \sum_{1 \leq i \leq 3, u \in \{i,j,k\}}{v_i(1)v_{i+1}(u)v_{i+2}(u)}\right) - \frac{1}{4} \det(v_j(u)). \nonumber \end{align} 

Now $\mathcal D$ is a product of terms of the form $D(x) := h \overline{L}(-x) h^{-1}$, for $x \in J$.  Let us compute the action of such a term on $\langle r(i), f g \rangle $ and $\langle r(-i), f g \rangle$.
\begin{lemma} $D(x) \langle r(i), f g \rangle  = 0$.  $D(x) \langle r(-i), f g \rangle  = -8i\langle (0,0,x,0), f g h \rangle$. \end{lemma}
\begin{proof} We have 
\[\langle r(i), f g he^{t\overline{L}(-x)}h^{-1} \rangle = \langle e e^{t\overline{L}(x)}, f g h \rangle = \langle e, f g h \rangle\]
is independent of $t$, proving the first claim.  Similarly,
\[\langle r(-i), f g he^{t\overline{L}(-x)}h^{-1} \rangle = -8i\langle f e^{t\overline{L}(x)}, f g h \rangle = -8i\langle f, f gh \rangle - 8i t \langle (0,0,x,0), f g h \rangle + O(t^2) \]
proving the second statement. \end{proof}
We now compute $D(z)D(y)D(x) e^{-|J(g,i)|^2}$.  First, since $|J(g,i)|^2 = \langle r(i), f g \rangle \langle r(-i), f g \rangle$,
\[D(x) e^{-|J(g,i)|^2} = e^{-|J(g,i)|^2}J(g,i)(8i \langle (0,0,x,0), f g h \rangle) = e^{-|J(g,i)|^2}\langle (0,0,x,0), v \rangle\]
where $v = 8i J(g,i) f g h$. Now, we have $D(y) \langle (0,0,x,0), f g h \rangle = \langle (0, x\times y, 0,0), fg h \rangle. $  Continuing, we thus get
\[D(y) D(x) e^{-|J(g,i)|^2} = e^{-|J(g,i)|^2}\left\{\langle (0,0,y,0), v \rangle \langle(0,0,x,0), v \rangle + \langle (0, x\times y, 0,0), v \rangle\right\}.\]
Finally, 
\[D(z)D(y)D(x) e^{-|J(g,i)|^2} = e^{-|J(g,i)|^2}\left(D_1(x,y,z)(v) + D_2(x,y,z)(v) + D_3(x,y,z)(v)\right),\]
where, for $x,y,z \in J$ and $w \in W$,
\begin{align*} D_1(x,y,z)(w) &= \langle (0,0,x,0),w \rangle \langle (0,0,y,0),w \rangle \langle (0,0,z,0),w \rangle, \\ D_2(x,y,z)(w) &=
\langle (0,0,x,0), w \rangle \langle (0,y \times z,0,0), w \rangle + \langle (0,0,y,0), w \rangle \langle (0,z \times x,0,0), w \rangle \\ & \qquad + \langle (0,0,z,0), w \rangle \langle (0,x\times y,0,0), w \rangle, \\ D_3(x,y,z)(w) &= \langle (\tr(z,x\times y),0,0,0), w \rangle. \end{align*}

\begin{proposition} Write $\mathcal D_0 = \sum{\alpha \, x y z}$, $\alpha \in \R$, $x,y,z \in J$, as abstract expression for the right hand side of (\ref{D_0}).  Suppose $w = (a,b,c,d)$. Then
\[\sum{\alpha D_1(x,y,z)(w)} =  \det(b); \; \; \; \sum{\alpha D_2(x,y,z)(w)} = -3\tr(b,c); \; \; \; \sum{\alpha D_3(x,y,z)(w)} = 15 d.\]
\end{proposition}
\begin{proof} This is a straightforward but tedious calculation, so we omit it. \end{proof}
Recall that $v = 8i J(g,i) f g h =:(a,b,c,d)$.  We obtain
\[(8i)^{-1} h(- \mathcal D_0 )h^{-1} e^{-|J(g,i)|^2} = (8i)^{-1}e^{-|J(g,i)|^2} \left(\det(b) - 3 \tr(b,c) + 15 d\right).\] 
But now,
\[a = \langle -f, v \rangle = J(g,i) \langle r(-i),f g \rangle = |J(g,i)|^2\]
and
\[d = \langle e, v \rangle = 8i J(g,i) \langle r(i), f g \rangle = 8i J(g,i)^2.\]
Since $v=(a,b,c,d)$ is rank one, $\det(b) = a^2 d$ and $\tr(b,c) = 3ad$.  Thus
\[\det(b) - 3\tr(b,c) + 15 d  = 8i \left\{|J(g,i)|^4 J(g,i)^2 - 9 |J(g,i)|^2 J(g,i)^2 + 15 J(g,i)^2 \right\}.\]
Hence 
\[(2\pi i)^{-3} h \mathcal D_0 h^{-1}   e^{-|J(g,i)|^2} =\pi^{-3} \left\{|J(g,i)|^4 - 9 |J(g,i)|^2 + 15 \right\}J(g,i)^2 e^{-|J(g,i)|^2}.\]
Since all the differential operators act trivially on the $J(g,i)$, we get 
\[\mathcal{D}\left( J(g,i)^{2r}  e^{-|J(g,i)|^2} \right)= \pi^{-3}\left\{|J(g,i)|^4 - 9 |J(g,i)|^2 + 15 \right\}J(g,i)^{2r+2} e^{-|J(g,i)|^2},\]
proving equation (\ref{DPhi2r}), and thus Theorem \ref{DE2r}.

\appendix

\section{Relation with $\GU_6(B)$}\label{GU6}
As above, we assume our ground field $F$ is of characteristic zero, and $B$ denotes a quaternion algebra over $F$.  To emphasize the relationship to the quaternion algebra $B$, we write $\GG_B = \GG$ for the group of linear automorphisms of $W_B = F \oplus H_3(B) \oplus H_3(B) \oplus F$ that preserve the symplectic and quartic form, up to similitude.  Set $J_6$, as usual, to be the matrix $\mm{}{1_3}{-1_3}{}$.  Define $\GU_6(B)$ to be the pairs $(g, \nu) \in \GL_6(B) \times \GL_1$ that satisfy $g J_6 g^* = \nu(g) J_6$.  The group $\GG_B$ is closely related to $\GU_6(B)$.  In this section, we work out an explicit relationship.  The results of this section are not needed in the proof of Theorem \ref{MainThm}; we have included this section only to help orient the reader.

The explicit relationship between $\GG_B$ and $\GU_6(B)$ takes the following form.  We define a reductive $F$-group $\GS$, that satisfies the following two properties.
\begin{itemize}
\item There are natural, explicit maps $\GS \rightarrow \GG_B$ and $\GS \rightarrow \GU_6(B)$, that are surjective on $\overline{F}$ points, where $\overline{F}$ is an algebraic closure of $F$.
\item The kernels of the two maps from $\GS$ are (distinct) $\mu_2$'s in the center of $\GS$.\end{itemize}
To define $\GS$, we proceed as follows.  First, we make an octonion algebra $\Theta$ out of $B$ using the Cayley-Dickson construction.  Denote by $W_\Theta = F \oplus H_3(\Theta) \oplus H_3(\Theta) \oplus F$ the Freudenthal space made out of the cubic Jordan algebra of Hermitian $3 \times 3$ matrices over $\Theta$.  Next, we define a decomposition $W_\Theta = W_B \oplus W_6 \otimes B$.  This decomposition induces an action of $\GG_B \times \GU_6(B)$ on $W_\Theta$, with $\GG_B$ acting via its defining representation on $W_B$ and as the identity on $W_6 \otimes B$, and vice versa for $\GU_6(B)$.  We define $\GS$ to be the set of pairs $(g,h) \in \GG_B \times \GU_6(B)$ that preserve the symplectic and quartic form on $W_\Theta$, up to similitude.  The natural maps $\GS \rightarrow \GG_B$ and $\GS \rightarrow \GU_6(B)$ are then the projections.

\subsection{The Cayley-Dickson construction} Fix a $\gamma \in \GL_1(F)$.  Then the octonion algebra $\Theta$ is defined to be the pairs $(x,y)$, $x, y$ in $B$, with addition defined component-wise, and multiplication as
\[ (x_1 , y_1) (x_2, y_2) = (x_1 x_2 + \gamma y_2^* y_1, y_2 x_1 + y_1 x_2^*).\]
The conjugation on $\Theta$ is $(x,y)^* = (x^*, -y)$.  If $z = (x,y)$ is in $\Theta$, then $\tr(z):= z+z^* = \tr(x)$ and $n(z) = z z^* = n(x) - \gamma n(y)$.  The choice $\gamma = -1$ is usually made to define the positive-definite octonion algebra.

Recall that the multiplication in $\Theta$ is \emph{not} associative, but that one does have the identity $\tr((z_1 z_2) z_3) = \tr(z_1 (z_2 z_3))$ for $z_1, z_2, z_3 \in \Theta$.

\subsection{The space $W_\Theta$.} The map $B \rightarrow \Theta$, $b \mapsto (b,0)$, induces an obvious inclusion $W_B \rightarrow W_\Theta$.  We define $\iota: W_6 \otimes B \rightarrow W_\Theta$ as follows.  Suppose $w \in W_6 \otimes B$ is $b_1 e_1 + b_2 e_2 + b_3 e_3 + c_1 f_1 + c_2 f_2 + c_3 f_3$, with $b_i, c_j \in B$.  Then $\iota(w)$ is defined to be $(0,B,C,0) \in W_\Theta$, with
\begin{equation}\label{BC} B = \left(\begin{array}{ccc} & (0,-b_3) & (0,b_2) \\ (0,b_3) & & (0,-b_1) \\ (0,-b_2) & (0,b_1) & \end{array}\right) \; \text{ and } \; C = \left(\begin{array}{ccc} & (0,c_3) & (0,-c_2) \\ (0,-c_3) & & (0,c_1) \\ (0,c_2) & (0,-c_1) & \end{array}\right).\end{equation}
(The difference in minus signs is intentional.)  It is clear that $W_\Theta = W_B \oplus \iota(W_6 \otimes B)$.

We now list some useful formulas.  Denote by $\langle \quad , \quad \rangle_B$ the $B$-valued symplectic-Hermitian form on $W_6 \otimes B$ preserved by $\GU_6(B)$.  That is, if $w = \sum_i{b_i e_i} + \sum_i{c_i f_i}$, and $w' = \sum_i{b'_i e_i} + \sum_i{c'_i f_i}$, then $\langle w, w' \rangle_B = \sum_{i}{b_i (c_i')^*} - \sum_i{c_i (b_i')^*}$.
\begin{itemize}
\item With $B, C$ as in (\ref{BC}), $\tr(B,C) = \gamma\left( (b_1,c_1) + (b_2,c_2) + (b_3,c_3)\right).$
\item The symplectic form on $W_\Theta$ restricts to the form $ (-\gamma) \langle \, , \, \rangle \otimes (\, , \, )$ on $W_6 \otimes B$.  That is, if $w = \sum_i{b_i e_i} + \sum_i{c_i f_i}$, and $w' = \sum_i{b'_i e_i} + \sum_i{c'_i f_i}$, then 
\begin{align}\label{ww'} \langle \iota(w), \iota(w') \rangle  &= -\gamma \left( \sum_i{(b_i,c'_i)} - \sum_i{(c_i, b'_i)}\right)\nonumber \\ &= -\gamma \tr(\langle w, w' \rangle_B). \end{align}
\item If
\[Y = \left(\begin{array}{ccc} & (0,y_3) & (0,-y_2) \\ (0,-y_3) & & (0,y_1) \\ (0,y_2) & (0,-y_1) & \end{array}\right),\]
then
\[Y^\# = \gamma \left( \begin{array}{c} y_1^* \\ y_2^* \\ y_3^* \end{array}\right) \left(\begin{array}{ccc} y_1 & y_2 & y_3 \end{array}\right) = \gamma \left(\begin{array}{ccc} n(y_1) & y_1^* y_2 & y_1^* y_3 \\ y_2^* y_1 & n(y_2) & y_2^* y_3 \\ y_3^* y_1 & y_3^* y_2  & n(y_3) \end{array}\right).\]
\item $Q(\iota(w)) = -\gamma^2 N( \langle w, w \rangle_B)$.  Indeed, suppose $w = (0,B,C,0)$, with $B, C$ as in (\ref{BC}).  Define $v_b = (b_1, b_2, b_3)$ and define $v_c$ similarly.  Then
\begin{align*} \gamma^{-2} Q(\iota(w)) &= \gamma^{-2}\tr(B,C)^2 - 4 \gamma^{-2} \tr(B^\#, C^\#) \\ &= \left(\sum_i{(b_i,c_i)}\right)^2 - 4\tr(v_b^* v_b, v_c^* v_c) \\ &= \tr( v_b v_c^*)^2 - 4 n( v_b v_c^*) \\ &= -n(v_b v_c^* - v_c v_b^*) \\ &= - n(\langle w, w \rangle_B). \end{align*}
\item For $b = (b_1, b_2, b_3)$ and $c = (c_1,c_2,c_3)$ in $B^3$, let us write $\iota_e(b) = B$, $\iota_f(c) = C$, with $B, C$ as in (\ref{BC}).  Then if $X, Y$ in $J = H_3(B)$, 
\begin{equation}\label{bX} X \times \iota_e(b) = \iota_f( b X)\; \text{ and } \; Y \times \iota_f(c) = \iota_e(cY).\end{equation}
\item Suppose $v \in W_B$, $\delta \in \iota(W_6 \otimes B)$.  Then $(v,v,v,\delta) = 0$ and $(v,\delta,\delta,\delta) = 0$.  Suppose $v = (a,b,c,d)$ and $\delta = (0,B,C,0)$. Then $(v,v,\delta,\delta)$ is proportional to (by a nonzero constant)
\begin{equation}\label{quad} \tr(B,C)\left(ad - \tr(b,c)\right) + \tr(b \times B, c \times C) + 2\tr(B^\#,c^\# - d b) + 2\tr(C^\#, b^\# - ac).\end{equation}
\end{itemize}

\subsection{The group $\GS$} We now have enough formulae to easily prove the facts about $\GS$ claimed above.  First note that by (\ref{ww'}), $(g,h) \in \GG_B \times \GU_6(B)$ preserves the symplectic form on $W_\Theta$ if and only if $\nu(g) = \nu(h)$.  

Now we construct some explicit elements in $\GS$. First, if $X \in H_3(B)$, then one computes $n_\Theta(X) = (n_B(X), \mm{1}{X}{}{1})$, and thus $n_\Theta(X)$ is an element of $\GS$.  The equality $n_\Theta(X) = (n_B(X), \mm{1}{X}{}{1})$ follows from (\ref{bX}).  Next, one verifies immediately that the following are all elements of $\GS$:
\begin{itemize}
\item $J_\Theta = (J_B, \mm{}{1_3}{-1_3}{})$.  Here $J_B, J_\Theta$ act on $W_B$, respectively $W_\Theta$, as $(a,b,c,d) \mapsto (-d,c,-b,a)$.
\item $m_\Theta(\lambda) = (m_B(\lambda),\mm{\lambda}{}{}{1})$.  Here $m_B(\lambda), m_\Theta(\lambda)$ act on $W_B$, respectively $W_\Theta$, as $(a,b,c,d) \mapsto (\lambda^2 a,\lambda b,c,\lambda^{-1}d)$.
\item $z_\Theta(\lambda) = (z_B(\lambda),\mm{\lambda}{}{}{\lambda})$.  Here $z_B(\lambda), z_\Theta(\lambda)$ act on $W_B$, respectively $W_\Theta$, as $(a,b,c,d) \mapsto \lambda (a,b,c,d)$. \end{itemize}
Now suppose $m \in \GL_3(B)$, and $N(m^* m)$ is a square in $F^\times$, say $N(m^* m ) = \delta^2$.  Associated to such a pair $(m,\delta)$, we construct an element $(M(m,\delta), \mm{m}{}{}{\,^*m^{-1}})$ in $\GS$.  To do this, define $M(m, \delta)$ by
\[(a,b,c,d) M(m, \delta) = (\delta a, \delta m^{-1} b \,^*m^{-1}, \delta^{-1} m^* c m, \delta^{-1} d).\]
If $b \in H_3(B)$, then $N( \delta m^{-1} b \,^*m^{-1}) = \delta^3 N(m^* m)^{-1} N(b) = \delta N(b)$, and hence $M(m, \delta)$ defines an element of $\GG^1_B$.  To check that $(M(m,\delta), \mm{m}{}{}{\,^*m^{-1}})$ defines an element of $\GS$, we must see that the map
\[ X + \iota_f(c) \mapsto \delta^{-1} m^* X m + \iota_f(c \,^*m^{-1})\]
on $H_3(\Theta)$ multiples the norm by $\delta^{-1}$.  But indeed, $N(X + \iota_f(c)) = N(X) + \tr(X, \iota_f(c)^\#)$, and then
\begin{align*} N(\delta^{-1} m^* X m + \iota_f(c \,^*m^{-1})) & = N(\delta^{-1} m^* X m) + \tr( \delta^{-1} m^* X m, \iota_f(c \,^*m^{-1})^\#) \\ &= \delta^{-1}N(X) + \delta^{-1} \tr(m^* X m,  m^{-1} v_c^* v_c \,^*m^{-1}) \\ &= \delta^{-1} N(X) + \delta^{-1} \tr(X, v_c^* v_c) \\ &= \delta^{-1} N(X + \iota_f(c)).\end{align*}
The explicit constructions now prove that $\GS \rightarrow \GG_B$ and $\GS \rightarrow \GU_6(B)$ are surjections on $\overline{F}$-points.

Let us now prove that the kernels of the maps $\GS \rightarrow \GG_B$ and $\GS \rightarrow \GU_6(B)$ are $\mu_2$'s.  First, one checks that easily that $\mu_2 \times \mu_2 \subseteq \GS$ via $(\pm 1, \pm 1)$, with the $\pm 1$'s acting by scalar multiplication on $W_B$ and $W_6 \otimes B$.  Hence a $\mu_2$ is contained in the each of the kernels.  Now for the converse.

First consider the map $\GS \rightarrow \GG_B$, and suppose $(1,h)$ is in the kernel.  Then $\nu(h) = 1$, and the element $(1,h)$ of $\GG_\Theta$ fixes $(1,0,0,0)$ and $(0,0,0,1)$.  It follows that it preserves the spaces $(0,*,0,0)$ and $(0,0,*,0)$, and thus $h = \mm{m}{}{}{\,^*m^{-1}}$ for some $m \in \GL_3(B)$.  Furthermore, we must have that $N(X + \iota_e(b)) = N(X+ \iota_e(b m))$ for all $X$ in $H_3(B)$ and $b \in B^3$, and thus
\[\tr(X, \iota_e(bm)^\#) = \tr(X, m^* v_b^* v_b m) = \tr(m X m^*, v_b^* v_b)\]
is preserved for all $b, X$.  Since the rank one elements of the form $v_b^* v_b$ span $H_3(B)$, we have $m X m^* = X$ for all $X \in H_3(B)$, and thus $m = \pm 1$, as desired.

Now consider the map $\GS \rightarrow \GU_6(B)$, and suppose $(g,1)$ is in the kernel.  Again we have $\nu(g) = 1$.  We will show $g = \pm 1 $ by considering the fact that $(g,1)$ preserves $(v,v,\delta, \delta)$.  By taking $C = 0$ and letting $B$ vary in (\ref{quad}), we see $g$ fixes the quantity $c^\# - db$, since the rank one elements $B^\# = v_b^* v_b$ span $H_3(B)$.  Similarly, by taking $B= 0$ in (\ref{quad}), we see $g$ fixes the quantity $b^\# - ac$.

Set $f =(0,0,0,1)$, and suppose $f g = (A,B,C,D)$, $(a,b,c,d)g = (a',b',c',d')$.  Then for any $\lambda \in F$ we get
\begin{align*} b^\# - ac &= (b' + \lambda B)^\# - (a' + \lambda A)(c' + \lambda C) \\ &= (b')^\# - a'c' + \lambda(b' \times B - a' C - A c'), \end{align*}
and thus $b' \times B - a' C - A c' = 0$.  By varying $(a,b,c,d)$, we can make $a', b', c'$ arbitrary.  It follows that $A = B = C = 0$.  Similarly, since $g$ fixes $c^\# - db$, $g$ preserves the line $F (1,0,0,0)$.  Hence, $(1,0,0,0) g = \lambda (1,0,0,0)$,  and $(0,0,0,1) g = \lambda^{-1}(0,0,0,1)$ for some $\lambda \in F^\times$.  It follows that $g$ preserves the spaces $(0,*,0,0)$ and $(0,0,*,0)$ in $W_\Theta$ as well, and scales the norm on $H_3(\Theta) = (0,*,0,0)$ by $\lambda$.  Write $X \mapsto t(X)$ for the action on $g$ on $H_3(B) = (0,*,0,0)$ in $W_B$.  Then we have $\tr(t(X),v_b^* v_b) = \lambda \tr(X, v_b^* v_b)$ for all $X \in H_3(B)$, $b \in B^3$.  Hence $t(X) = \lambda X$.  But since $t$ also scales the norm by $\lambda$, we get $\lambda^3 = \lambda$, and hence $\lambda = \pm 1$, as desired.  This completes the proof of the facts about $\GS$.

\section{The magic triangle of Deligne and Gross}\label{magic}
In this section we very briefly describe the magic triangle of Deligne and Gross \cite{dG}, and its mysterious apparent connection to some Rankin-Selberg integrals. This possible connection appears to be closely related to the so-called \emph{towers} of Rankin-Selberg integrals introduced by Ginzburg-Rallis \cite{ginzRallis}.  That there may be a connection between the Freudenthal magic square and Rankin-Selberg integrals is discussed, from a different point of view, in Bump's survey article \cite{bumpRS}.

The magic triangle is an extension of the Freudenthal magic square.  Roughly speaking, each row of the triangle contains a sequence of groups $H_1 \subseteq H_2 \subseteq H_3 \subseteq ....$, each group included in the next.  And the groups in the same row all share similar properties.  The final row of the triangle is the exceptional series of Deligne \cite{deligne1}, \cite{deligneDeMan}:
\begin{equation}\label{adjoint} A_1 \subset A_2 \subset G_2 \subset D_4 \subset F_4 \subset E_6 \subset E_7 \subset E_8.\end{equation}
The penultimate row of the triangle is the one most relevant for this paper:
\begin{equation}\label{Freudenthal} \mathbf{G}_m \subset A_1 \subset A_1^3 \subset C_3 \subset A_5 \subset D_6 \subset E_7.\end{equation}
Each group in the triangle has a preferred representation.  Many of the groups also have a preferred parabolic subgroup, obtained by taking the stabilizer of a highest line in the preferred representation.  In row (\ref{adjoint}), the preferred representation is the adjoint representation, and the preferred parabolic subgroup is a Heisenberg parabolic.  In row (\ref{Freudenthal}), the preferred representation is given by Freudenthal's space $F \oplus J \oplus J \oplus F$, where $J$ is a certain cubic Jordan algebra that corresponds to the group in question.  The preferred parabolic subgroup is the stabilizer of a rank one line in this space, which is a Siegel parabolic.  So, the group $D_6$ in row (\ref{Freudenthal}) is the group $\GG$ from the main body of the paper, its preferred representation being $W$ and the preferred parabolic subgroup being the Siegel parabolic $P$, the parabolic used to define the Eisenstein series in the integral representation.

One can find many Rankin-Selberg integrals by applying the following procedure.  Take two groups $H_i \subset H_j$ or $H_j \subset H_i$ in the same row.  Take a cusp form $\phi_i$ on $H_i(\A)$, and an Eisenstein series $E_j(g,s)$ corresponding to the preferred parabolic subgroup on $H_j(\A)$.  Then one gets a global Rankin-Selberg convolution
\begin{equation} \label{RS1} \int_{H_i(\Q)\backslash H_i(\A)}{\phi_i(g) E_j(g,s)\,dg}\end{equation}
if $H_i \subset H_j$ or
\begin{equation}\label{RS2}\int_{H_j(\Q)\backslash H_j(\A)}{\phi_i(g) E_j(g,s)\,dg}\end{equation}
if $H_j \subset H_i$.  Sometimes one must add in some additional unipotent integration to make the Rankin-Selberg convolution produce an $L$-function.  Of course, considering the inclusion $C_3 \subset D_6$ yields the integral in this paper, whereas the inclusion $A_1^3 \subset C_3$ yields the integral of \cite{pollackShah}, if one puts the cusp form on $C_3$ and the Eisenstein series on $A_1^3$.  If one instead puts the cusp form on $A_1^3$ and the Eisenstein series on $C_3$, one gets the triple product $L$-function of Garrett \cite{garrett} and Piatetski-Shapiro--Rallis \cite{psrRankin3}.  The inclusion $A_1 \subset C_3$ yields the Whittaker integral of Bump-Ginzburg \cite{bg} for the Spin $L$-function on $\GSp_6$, except now one must add a unipotent group to the $A_1$ here.

The integral of Furusawa for the degree eight $L$-function on $\GSp_4 \times \GU(1,1)$ is not technically present in the row (\ref{Freudenthal}), but one can add it there, as follows.  The groups, their preferred representations, and their preferred parabolic subgroups in this row all come from the Freudenthal construction.  The Freudenthal construction can be made for a general cubic Jordan algebra $J$.  For example, taking $J$ to be $H_3(B)$ yields the group $\GG$ of type $D_6$, whereas taking $J$ to be $Sym_3(\Q)$, the $3 \times 3$ symmetric matrices over $\Q$, yields the group $\GSp_6$.  If $E$ is a quadratic imaginary extension of $\Q$, taking $J = H_3(E)$, the three-by-three Hermitian matrices over $E$, yields the group $A_5$ from (\ref{Freudenthal}), which is essentially $\GU(3,3)$.  If one now takes $J = \Q \oplus Sym_2(\Q)$, then the Freudenthal construction produces a group $H_J$ closely related to $\GU(1,1) \times \GSp_4$.  Thus, one can think of the inclusion $H_J \subset \GU(3,3)$ as sitting in the row (\ref{Freudenthal}) as well, and this gives the integral of Furusawa.

Amazingly, at least in the row (\ref{Freudenthal}), the magic triangle also helps to predict the $L$-function realized by the integrals (\ref{RS1}) and (\ref{RS2}).  The groups in row (\ref{Freudenthal}) appear as Levi subgroups of the preferred parabolic subgroup of a corresponding group from row (\ref{adjoint}).  In the case of $C_3$, the inclusion is into the Levi of the Heisenberg parabolic subgroup of the group $F_4$ from the exceptional series.  The Spin $L$-function on $C_3$ appears in the constant term of a cuspidal Eisenstein series for this parabolic on $F_4$.

While we have mentioned some Rankin-Selberg integrals that seem to be connected to the row (\ref{Freudenthal}), one can find more Rankin-Selberg integrals in the literature that appear to be connected to this row, and to the other rows of the triangle.  We also mention that the magic triangle appears both in the literature on dual pairs and the minimal representations of exceptional groups, see Rumelhart \cite{rumelhart}, and in recent works on arithmetic invariant theory, see Bhargava-Ho \cite{bhargavaHo}.
\bibliography{integralRepnBibNEW} 
\bibliographystyle{abbrv}
\end{document}